\documentclass{amsart}
\usepackage{amscd}
\usepackage{amssymb}
\usepackage[all]{xy}
\usepackage{mathrsfs}
\usepackage{upgreek}
\usepackage{stmaryrd}
\usepackage[colorlinks = true]{hyperref}
\usepackage{rotating}
\usepackage{tikz}
\usepackage{tikz-cd}
\usepackage{bbm}
\usetikzlibrary{matrix,arrows,cd,decorations.pathmorphing}
\usepackage{mathrsfs}
\usepackage{centernot} 
\usepackage{mathtools}
\usepackage{extarrows} 

\DeclareMathAlphabet\mathbfcal{OMS}{cmsy}{b}{n}

\newcommand{\bk}{\Bbbk}

\newcommand{\Z}{\mathbb{Z}}
\newcommand{\C}{\mathbb{C}}

\newcommand{\Gm}{\mathbb{G}_{\mathrm{m}}}

\newcommand{\prodrat}{\sideset{}{^{\mathrm{rat}}}\prod}
\newcommand{\smallprodrat}{\prod^{\mathrm{rat}}}

\newcommand{\fg}{\mathfrak{g}}
\newcommand{\fb}{\mathfrak{b}}

\newcommand{\fn}{\mathfrak{n}}

\newcommand{\fp}{\mathfrak{p}}

\newcommand{\weyl}{\mathsf{M}}
\newcommand{\coweyl}{\mathsf{N}}
\newcommand{\tilt}{\mathsf{T}}

\newcommand{\St}{\mathrm{St}}

\newcommand{\cN}{\mathcal{N}}

\newcommand{\tcN}{{\widetilde{\mathcal{N}}}}

\newcommand{\bX}{\mathbf{X}}

\newcommand{\bXp}{\mathbf{X}^+}
\newcommand{\ext}{{\mathrm{ext}}}

\newcommand{\bXpp}{\bX^{+,\mathrm{reg}}}
\newcommand{\dom}{\mathsf{dom}}

\newcommand{\Waff}{W_{\mathrm{aff}}}

\newcommand{\cF}{\mathcal{F}}
\newcommand{\cG}{\mathcal{G}}

\newcommand{\cO}{\mathcal{O}}

\newcommand{\cS}{\mathcal{S}}
\newcommand{\cI}{\mathcal{I}}

\newcommand{\tnabla}{\widetilde{\nabla}}
\newcommand{\tDelta}{\widetilde{\Delta}}

\newcommand{\ocA}{\overline{\mathcal{A}}}
\newcommand{\onabla}{\overline{\nabla}}
\newcommand{\oDelta}{\overline{\Delta}}

\newcommand{\bS}{\mathbf{\Sigma}}

\newcommand{\fA}{\mathfrak{A}}
\newcommand{\fC}{\mathfrak{C}}
\newcommand{\cC}{\mathcal{C}}

\newcommand{\lmod}{\text{-}\mathsf{mod}}
\newcommand{\ldgmod}{\text{-}\mathsf{dgmod}}

\newcommand{\Rep}{\mathsf{Rep}}

\newcommand{\Coh}{\mathsf{Coh}}
\newcommand{\QCoh}{\mathsf{QCoh}}

\newcommand{\Tilt}{\mathsf{Tilt}}

\newcommand{\fD}{\mathfrak{D}}
\newcommand{\cX}{\mathcal{X}}
\newcommand{\exto}{{\mathrm{ext},\oplus}}
\newcommand{\tri}{{\mathrm{tri}}}
\newcommand{\trio}{{\mathrm{tri},\oplus}}
\newcommand{\lgen}{\langle \!\langle}
\newcommand{\rgen}{\rangle \!\rangle}
\newcommand{\uast}{\mathbin{\underline{*}}}
\newcommand{\fS}{\mathfrak{S}}

\newcommand{\Db}{D^{\mathrm{b}}}

\DeclareMathOperator{\Hom}{Hom}
\DeclareMathOperator{\Ext}{Ext}

\DeclareMathOperator{\End}{End}
\DeclareMathOperator{\Sym}{Sym}
\DeclareMathOperator{\Ind}{Ind}
\DeclareMathOperator{\Res}{Res}

\DeclareMathOperator{\pr}{pr}

\DeclareMathOperator{\supp}{supp}
\DeclareMathOperator{\im}{im}
\DeclareMathOperator{\Spec}{Spec}

\newcommand{\fgen}{{\mathrm{fg}}}

\newcommand{\id}{\mathrm{id}}
\newcommand{\simto}{\xrightarrow{\sim}}
\newcommand{\la}{\langle}
\newcommand{\ra}{\rangle}
\newcommand{\lb}{\mathord{\lbag}}
\newcommand{\rb}{\mathord{\rbag}}

\newcommand{\ol}[1]{\overline{#1}}

\newcommand{\mbf}[1]{\mathbf{#1}}

\newcommand{\il}{\iota^{\mathrm{L}}}
\newcommand{\ir}{\iota^{\mathrm{R}}}
\newcommand{\pil}{\Pi^{\mathrm{L}}}
\newcommand{\pir}{\Pi^{\mathrm{R}}}
\newcommand{\Bru}{{\mathrm{Bru}}}

\newcommand{\fE}{\mathfrak{E}}
\newcommand{\tfL}{\widetilde{\mathfrak{L}}}
\newcommand{\fL}{\mathfrak{L}}
\newcommand{\tfE}{\widetilde{\mathfrak{E}}}

\makeatletter
\def\lotimes{\@ifnextchar_{\@lotimessub}{\@lotimesnosub}}
\def\@lotimessub_#1{\mathchoice{\mathbin{\mathop{\otimes}^L}_{#1}}%
  {\otimes^L_{#1}}{\otimes^L_{#1}}{\otimes^L_{#1}}}
\def\@lotimesnosub{\mathbin{\mathop{\otimes}^L}}
\makeatother

\numberwithin{equation}{section}
\newtheorem{thm}{Theorem}[section]
\newtheorem{lem}[thm]{Lemma}
\newtheorem{prop}[thm]{Proposition}
\newtheorem{cor}[thm]{Corollary}
\newtheorem{conj}[thm]{Conjecture}
\theoremstyle{definition}
\newtheorem{defn}[thm]{Definition}

\theoremstyle{remark}
\newtheorem{rmk}[thm]{Remark}

\title[Co-t-structures on derived categories of coherent sheaves]{Co-t-structures on derived categories of coherent sheaves and the cohomology of tilting modules}

 \author{Pramod N. Achar}
 \address{Department of Mathematics\\
   Louisiana State University\\
   Baton Rouge, LA 70803\\
   U.S.A.}
 \email{pramod@math.lsu.edu}

 \author{William Hardesty}
 \address{School of Mathematics and Statistics\\
   University of Sydney\\
   Camperdown, NSW 2006\\
   U.S.A.}
 \email{hardes11@gmail.com}
 
 \thanks{P.A. was supported by NSF Grant No.~DMS-1802241. W.H. was supported by the ARC Discovery Grant No.~DP170104318.}
 \dedicatory{Dedicated to the memory of Jim Humphreys}

\begin{document}

 \begin{abstract}
We construct a co-$t$-structure on the derived category of coherent sheaves on the nilpotent cone $\cN$ of a reductive group, as well as on the derived category of coherent sheaves
on any parabolic Springer resolution. 
These structures are employed to show that the push-forwards of the ``exotic parity objects'' (considered in \cite{achar-hardesty-riche}), along the (classical) Springer resolution, give indecomposable objects inside the 
coheart of the co-$t$-structure on $\cN$. 
 We also demonstrate how the various parabolic co-$t$-structures can be related by introducing an analogue to the usual translation functors. As an application, we give a proof of a scheme-theoretic formulation of the relative Humphreys conjecture on support varieties of tilting modules in type $A$ for $p>h$.

\end{abstract}
\maketitle

\section{Introduction}

Let $\mbf{G}$ be a connected reductive group over an algebraically closed field $\bk$ of characteristic $p > h$ (where $h$ is the Coxeter number for $\mbf{G}$), and let $\mbf{G}_{1}$ be its first Frobenius kernel.  This paper is concerned with the study of the $\mbf{G}_{1}$-cohomology of $\mbf{G}$-modules, denoted by $H^\bullet(\mbf{G}_{1},M)$.  Via the well-known identification of $H^\bullet(\mbf{G}_{1},\bk)$ with the coordinate ring of the nilpotent cone $\cN$ of the Frobenius twist of $\mbf{G}$ (see~\cite{aj,fp}), we see that $H^\bullet(\mbf{G}_{1},M)$ can be thought of as a coherent sheaf on $\cN$.

The \emph{relative Humphreys conjecture}~\cite{hum:cmr} describes the support of this coherent sheaf when $M$ is an indecomposable tilting $\mbf{G}$-module.  This conjecture has been proved for $\mbf{G} = \mathrm{GL}_n$ when $p > n$ by the second author~\cite{hardesty}, and for arbitrary reductive $\mbf{G}$ when $p \gg 0$ by the authors together with S.~Riche~\cite{achar-hardesty-riche}.  The latter paper drew inspiration from Bezrukavnikov's work on the quantum group version of the Humphreys conjecture~\cite{bez-tilting}, especially in its invocation of derived equivalences with constructible sheaves on affine flag varieties.

However, one key feature of Bezrukavnikov's work was absent from~\cite{achar-hardesty-riche}. In~\cite{bez-tilting}, Bezrukavnikov gave an intrinsic characterization of the (complexes of) coherent sheaves on $\cN$ that can arise as the cohomology of tilting modules for a quantum group: they are precisely the simple objects in the heart of the \emph{perverse-coherent $t$-structure} on $\cN$.  This characterization is false in the reductive group case considered in~\cite{achar-hardesty-riche}, and no alternative characterization was known at that time.

This paper was inspired by a desire to find such an intrinsic characterization.  Instead of using $t$-structures, we equip the derived category of (dg) coherent sheaves on $\cN$ with a \emph{co-$t$-structure}, and we show (see Lemma~\ref{lem:cohom-silting}) that the indecomposable objects in its coheart (called \emph{silting objects}) are precisely the sheaves arising as $\mbf{G}_{1}$-cohomology of indecomposable tilting $\mbf{G}$-modules.

But the main results have to do with the interaction between our co-$t$-structure and (modified) translation functors between different blocks of $\mbf{G}$-modules.  Of course, $H^\bullet(\mbf{G}_{1},M) = \Ext^\bullet_{\mbf{G}_{1}}(\bk,M)$ is zero unless $M$ lies in the (extended) principal block of $\mbf{G}$.  But if $M$ lies in some other (typically singular) block, one can instead study $\Ext^\bullet_{\mbf{G}_{1}}(\St_I^{\mbf{G}},M)$, where $\St_I^{\mbf{G}}$ is a suitable ``Steinberg-type module'' (see Section~\ref{sec:repG} for the definition).  These $\Ext$-groups still have the structure of a module over $\Ext^\bullet_{\mbf{G}_{1}}(\bk,\bk) \cong \bk[\cN]$, and hence give coherent sheaves on $\cN$.

The main technical result of the paper, Theorem~\ref{thm:pushforward-stein-commute}, implies that the module $\Ext^\bullet_{\mbf{G}_{1}}(\St_I^{\mbf{G}},M)$ again lies in the coheart of our co-$t$-structure.  The support of $\Ext^\bullet_{\mbf{G}_{1}}(\St_I^{\mbf{G}},M)$ is easily seen to be contained in the closure of the Richardson orbit corresponding to the set $I$ of simple reflections.  Thus, Theorem~\ref{thm:pushforward-stein-commute} implies that there is a large supply of silting objects supported on Richardson orbit closures.

As an application, in Section~\ref{sec:scheme-hum}, we give a new, geometric proof of the relative Humphreys conjecture for $\mathrm{GL}_n$, valid for $p > n$.  This proof hinges on the fact that every nilpotent orbit for $\mathrm{GL}_n$ is Richardson.

In fact, we obtain a stronger statement than what was known before: we show that if $M$ is an indecomposable tilting module in the principal block for $\mathrm{GL}_n$, then $H^\bullet(\mbf{G}_{1},M)$ is supported \emph{scheme-theoretically} (not just set-theoretically) on the orbit closure predicted by Humphreys.  We call this stronger statement the \emph{scheme-theoretic (relative) Humphreys conjecture}.  We expect that this scheme-theoretic version holds for any reductive group $\mbf{G}$, and we hope that the tools developed in this paper may be useful for proving it.

\subsection*{Contents}

Section~\ref{sec:background} contains preliminaries on co-$t$-structures and on (co-)quasi-exceptional sequences.  In Sections~\ref{sec:cotangent} and~\ref{sec:nilpotent}, we apply this machinery to the derived category of coherent sheaves on the cotangent bundle of a partial flag variety (denoted by $\tcN_I$), and on the nilpotent cone (denoted by $\cN$), respectively. These sections define the \emph{supportive co-$t$-structure}, which is a primary focus of this paper.  Next, in Section~\ref{sec:pushforward}, we introduce the setting of ``dg coherent sheaves'' on both $\tcN_I$ and $\cN$, and we show how to adapt some of the material from the preceding sections to this context.  (For additional background on dg coherent sheaves on affine schemes, see Appendix~\ref{app:dg}.)  Sections~\ref{sec:cotangent}--\ref{sec:pushforward} only assume that the characteristic $p$ of $\bk$ is ``pretty good''; it need not be larger than the Coxeter number.

Starting from Section~\ref{sec:repG}, we assume that $p > h$, and that the derived subgroup of $\mbf{G}$ is simply connected.  In Section~\ref{sec:repG}, we establish the relationship between the supportive co-$t$-structure and the $\mbf{G}_{1}$-cohomology of tilting modules.  The main result of this section relates this co-$t$-structure to the groups $\Ext^\bullet_{\mbf{G}_{1}}(\St_I^{\mbf{G}},M)$ discussed above.  Finally, Section~\ref{sec:scheme-hum} contains a precise formulation of the scheme-theoretic Humphreys conjecture, as well as its proof in the case of $\mathrm{GL}_n$.

\subsection*{Added in revision}
Since this paper first appeared in preprint form, the authors have obtained~\cite{ah2} a proof of the (set-theoretic) relative Humphreys conjecture in general, by an argument that makes crucial use of the co-$t$-structure machinery developed in this paper.  (However, the scheme-theoretic version proposed in Section~\ref{sec:scheme-hum} remains open outside of $\mathrm{GL}_n$.)

\subsection*{Acknowledgments}

We are grateful to Linyuan Liu for a careful reading of this paper, and to an anonymous referee for very extensive comments that have helped substantially improve the exposition.

\section{Background on co-\texorpdfstring{$t$}{t}-structures}
\label{sec:background}

In this section, we consider various homological algebra constructions involving \emph{co-$t$-structures} (the definition will be reviewed below).  Most of the results in this section have close analogues for $t$-structures, found in~\cite{bbd,bez-quasi} among other sources.  From Section~\ref{ss:tate} on, we will emphasize the parallels between the $t$-structure and co-$t$-structure situations by including both settings in the statements of most lemmas and propositions below.  However, we will usually give proofs only in the co-$t$-structure case.

\subsection{Definition and generalities}

We begin with some notation for subcategories of triangulated categories.  Given a set of objects $\cX$ in a triangulated category $\fD$, we define four full subcategories of $\fD$ as follows:
\[
\begin{array}{l@{\qquad}l}
&
\text{\it\dots is defined to be the smallest full additive sub-} \\
\text{\it Notation}  &\text{\it category of $\fD$ containing $\cX$ and closed under \dots} \\
\hline
\la \cX \ra_\ext & \text{extensions} \\
\la \cX \ra_\exto & \text{extensions and direct summands} \\
\la \cX \ra_\tri & \text{$[\pm1]$ and extensions} \\
\la \cX \ra_\trio & \text{$[\pm1]$, extensions, and direct summands}
\end{array}
\]
The latter two are triangulated categories.  We obviously have $\la \cX \ra_\ext \subset \la \cX \ra_\exto$ and $\la \cX \ra_\tri \subset \la \cX \ra_\trio$; in some situations, these containments are equalities.

We now recall the definition and some basic facts on co-$t$-structures (see \cite{jorg, koenig-yang} for an overview).

\begin{defn}\label{defn:cotstruc}
Let $\fD$ be a triangulated category.  A \emph{co-$t$-structure} on $\fD$ is a pair of full additive subcategories $(\fD_{\ge 0}, \fD_{\le 0})$ with the following properties:
\begin{enumerate}
\item Both $\fD_{\ge 0}$ and $\fD_{\le 0}$ are closed under direct summands.\label{ax:thick}
\item We have $\fD_{\ge 0}[-1] \subset \fD_{\ge 0}$ and $\fD_{\le 0}[1] \subset \fD_{\le 0}$.\label{ax:shift}
\item For $A \in \fD_{\ge 0}$ and $B \in \fD_{\le 0}[1]$, we have $\Hom(A,B) = 0$.\label{ax:hom}
\item For any $X \in \fD$, there exists a distinguished triangle $A \to X \to B \to$ with $A \in \fD_{\ge 0}$ and $B \in \fD_{\le 0}[1]$.\label{ax:dt}
\end{enumerate}
(These axioms imply that $\fD_{\ge 0}$ and $\fD_{\le 0}$ are automatically closed under extensions.)
A co-$t$-structure $(\fD_{\ge 0}, \fD_{\le 0})$ is said to be \emph{bounded} if
\[
\bigcup_{n \in \Z} \fD_{\ge 0}[n] = \bigcup_{n \in \Z} \fD_{\le 0}[n] = \fD.
\]
The additive subcategory $\fC = \fD_{\ge 0} \cap \fD_{\le 0}$ is called the \emph{coheart} of the co-$t$-structure. Objects of $\fC$ are called \emph{silting objects}.  A \emph{silting generator} is a silting object $T$ with the property that every object in $\fC$ is a direct sum of direct summands of $T$.
\end{defn}


\begin{rmk}
Our definition of the term \emph{silting} is not consistent with the co-$t$-structure literature, and the difference in usage is similar to the difference in the usage of the word \emph{tilting} for algebraic groups or for finite-dimensional algebras.  In most of the literature, a \emph{silting object} is defined to be an object that generates a silting subcategory (see Definition~\ref{defn:silting-cat} below) under direct sums and direct summands.  
\end{rmk}

\begin{defn}
If $\fD$ and $\fD'$ are two triangulated categories equipped with co-$t$-struc\-tures with cohearts
$\fC \subset \fD$ and $\fC' \subset \fD'$, then we say that a triangulated functor 
$F:\fD \rightarrow \fD'$ is \emph{co-$t$-exact} if $F(\fC) \subset \fC'$ (i.e. $F$ preserves the cohearts). 
\end{defn}

Conversely, one can intrinsically characterize subcategories arising as cohearts of co-$t$-structures in the following way.

\begin{defn}\label{defn:silting-cat}
A full additive subcategory $\fS \subset \fD$ is a \emph{silting subcategory} if it satisfies the following properties:
\begin{enumerate}
\item $\fS$ is closed under direct summands.\label{it:silting-karoubi}
\item For any two objects $S$, $S' \in \fS$, we have $\Hom(S,S'[k]) = 0$ for all $k > 0$.\label{it:silting-ext}
\item $\fD = \la \fS\ra_\trio$.\label{it:silting-gen}
\end{enumerate}
\end{defn}

For a proof of the following proposition, see~\cite[Corollary~5.9]{hvvs:abcc}.
 
\begin{prop}\label{prop:silt_co-t-structure}
Let $\fD$ be a triangulated category.  A full subcategory $\fS \subset \fD$ is a silting subcategory if and only if it is the coheart of a bounded co-$t$-structure on $\fD$.  In this case, the co-$t$-structure is uniquely determined: it is given by
\[
\fD_{\geq 0} = \la \fS[-k] : k \geq 0 \ra_\exto, \qquad \fD_{\leq 0} = \la \fS[k] : k \geq 0 \ra_\exto. 
\]
\end{prop}
\begin{rmk}\label{rmk:natural-co-t-structure}
If $\fD = \Db\fA$ where $\fA$ is a highest weight category, then 
the subcategory $\Tilt(\fA) \subset \fA$ of tilting objects is a silting subcategory. 
The corresponding co-$t$-structure on $\fD$ will be called the \emph{natural co-$t$-structure}. 
(An alternative description of the co-$t$-structure in terms of standard and co-standard objects will be given in Proposition~\ref{prop:preexcep-struc}.)
\end{rmk}

The following proposition describes parallel ways of constructing $t$- and co-$t$-structures in a category generated by a single object.

\begin{prop}\label{prop:single-gen}
Let $\bk$ be a field, and let $\fD$ be a $\bk$-linear triangulated category.  Suppose there is an object $A$ that generates $\fD$ as a triangulated category, and assume that $\End(A) \cong \bk$. Let
\[
\fD(\le 0) = \la A[i] : i \ge 0 \ra_\exto
\qquad\text{and}\qquad
\fD(\ge 0) = \la A[i] : i \le 0 \ra_\exto.
\]
Then:
\begin{enumerate}
\item If $\Hom(A,A[i]) = 0$ for $i < 0$, then $(\fD(\le 0), \fD(\ge 0))$ is a $t$-structure on $\fD$.  In fact, it is the unique $t$-structure whose heart contains $A$.  Moreover, the heart is a finite-length abelian category, and $A$ is the unique simple object.
\item If $\Hom(A,A[i]) = 0$ for $i > 0$, then $(\fD(\ge 0), \fD(\le 0))$ is a co-$t$-structure on $\fD$.  In fact, it is the unique co-$t$-structure whose coheart contains $A$. Moreover, if $\fD$ is Karoubian, then the coheart is a Krull--Schmidt category, and $A$ is the unique indecomposable object.
\end{enumerate}
\end{prop}

We will not prove this proposition; instead, we will prove a ``graded'' variant of it later on (see Proposition~\ref{prop:gr-single-gen}).  The proof given there is easily adapted to prove Proposition~\ref{prop:single-gen}.  For the $t$-structure part of this statement, see~\cite[Lemma~3]{bez-quasi}.

\begin{lem}\label{lem:krull-schmidt}
Let $\bk$ be a field, and let $\fD$ be a $\bk$-linear triangulated category equipped with a co-$t$-structure $(\fD_{\ge 0}, \fD_{\le 0})$.  Suppose the following conditions hold:
\begin{enumerate}
\item The category $\fD$ is Karoubian.
\item For $X \in \fD_{\ge 0}$ and $Y \in \fD_{\le 0}$, the $\bk$-vector space $\Hom(X,Y)$ is finite-dimensional.\label{it:ks-findim}
\end{enumerate}
Then the coheart $\fC = \fD_{\ge 0} \cap \fD_{\le 0}$ is a Krull--Schmidt category.
\end{lem}
\begin{proof}
The coheart of our co-$t$-structure is closed under direct summands, so it is also Karoubian.  Assumption~\eqref{it:ks-findim} implies that $\Hom$-spaces in $\fC$ are finite-di\-men\-sion\-al.  According to~\cite[Corollary~A.2]{cyz:addz}, $\fC$ is Krull--Schmidt.
\end{proof}

The following proposition involves the notion of ``recollement'' from~\cite{bbd}.  In the diagram shown below, $\iota$ is fully faithful, and $\Pi$ identifies $\fD''$ with the Verdier quotient $\fD/\iota(\fD')$.  The functors $\il$ and $\pil$ are the left adjoints of $\iota$ and $\Pi$, respectively, while $\ir$ and $\pir$ are their right adjoints.  These six functors are assumed to satisfy some additional axioms, spelled out in~\cite[\S1.4.3]{bbd}.  For the $t$-structure analogue of the statement below, see~\cite[Th\'eor\`eme~1.4.10]{bbd}.

\begin{prop}\label{prop:recolle-cotstruc}
Let $\fD'$, $\fD$, and $\fD''$ be three triangulated categories, and suppose we have a recollement diagram
\[
\begin{tikzcd}
\fD' \ar[r, "\iota" description] &
\fD \ar[r, "\Pi" description] \ar[l, bend right, "\il"'] \ar[l, bend left, "\ir" ] &
\fD''. \ar[l, bend right, "\pil"'] \ar[l, bend left, "\pir" ]
\end{tikzcd}
\]
Assume that $\fD'$ is equipped with a co-$t$-structure $(\fD'_{\ge 0}, \fD'_{\le 0})$, and that $\fD''$ is equipped with a co-$t$-structure $(\fD''_{\ge 0}, \fD''_{\le 0})$.  Then the categories 
\begin{align*}
\fD_{\ge 0} &= \{ X \mid \il(X) \in \fD'_{\ge 0}, \Pi(X) \in \fD''_{\ge 0} \}, \\
\fD_{\le 0} &= \{ X \mid \ir(X) \in \fD'_{\le 0}, \Pi(X) \in \fD''_{\le 0} \}
\end{align*}
constitute a co-$t$-structure on $\fD$.
\end{prop}

For a proof, see~\cite[\S8.2]{bondarko}.  (In~\cite{bondarko}, co-$t$-structures are called \emph{weight structures}.)  The new co-$t$-structure on $\fD$ given by this proposition is said to be obtained from those on $\fD'$ and $\fD''$ by \emph{gluing} or \emph{recollement}.  Note that the functors $\iota$ and $\Pi$ are co-$t$-exact.

\begin{lem}\label{lem:recolle-full}
In the setting of the recollement diagram of Proposition~\ref{prop:recolle-cotstruc}, if $X \in \fD_{\ge 0}$ and $Y \in \fD_{\le 0}$, then the map
\begin{equation}\label{eqn:recolle-full}
\Pi: \Hom(X,Y) \to \Hom(\Pi(X),\Pi(Y))
\end{equation}
is surjective.  As a consequence, if the coheart $\fC = \fD_{\ge 0} \cap \fD_{\le 0}$ is a Krull--Schmidt category, then $\Pi$ sends any indecomposable object in $\fC$ to either $0$ or an indecomposable object.
\end{lem}
\begin{proof}
The recollement formalism gives us a distinguished triangle $\iota\ir(Y) \to Y \to \pir\Pi(Y) \to$, and hence a long exact sequence
\begin{equation}\label{eqn:recolle-full-les}
\cdots \to \Hom(X,Y) \to \Hom(X, \pir\Pi(Y)) \to \Hom(X,\iota\ir(Y)[1]) \to \cdots.
\end{equation}
We have $\Hom(X,\iota\ir(Y)[1]) \cong \Hom(\il(X), \ir(Y)[1])$.  From Proposition~\ref{prop:recolle-cotstruc}, we have $\il(X) \in \fD'_{\ge 0}$ and $\ir(Y) \in \fD'_{\le 0}$.  From the definition of a co-$t$-structure, we have $\Hom(\il(X), \ir(Y)[1]) = 0$, so the first map in~\eqref{eqn:recolle-full-les} is surjective.  Via the isomorphism $\Hom(X,\pir\Pi(Y)) \cong \Hom(\Pi(X), \Pi(Y))$, we see that~\eqref{eqn:recolle-full} is also surjective.

As a special case, for any object $X \in \fC$, the map $\End(X) \to \End(\Pi(X))$ is surjective.  In particular, if $\End(X)$ is a local ring and $\Pi(X)$ is nonzero, then $\End(\Pi(X))$ is also a local ring (since any quotient of a local ring is local), and hence $\Pi(X)$ is indecomposable.  The last assertion then follows, since every indecomposable object in a Krull--Schmidt category has a local endomorphism ring.
\end{proof}

\begin{prop}\label{prop:indec-class}
In the setting of the recollement diagram of Proposition~\ref{prop:recolle-cotstruc}, let $\fC'$, $\fC$, and $\fC''$ denote the cohearts of the co-$t$-structures on $\fD'$, $\fD$ and $\fD''$, respectively.  Assume that all three cohearts are Krull--Schmidt categories.
\begin{enumerate}
\item If $T' \in \fC'$ is indecomposable, then $\iota(T') \in \fC$ is indecomposable.\label{it:iota-indec}
\item If $T'' \in \fC''$ is indecomposable, there is an indecomposable object $T \in \fC$, unique up to isomorphism, such that $\Pi(T) \cong T''$.\label{it:pi-indec}
\item Every indecomposable object in $\fC$ comes from parts~\eqref{it:iota-indec} or~\eqref{it:pi-indec} above.\label{it:all-indec}
\end{enumerate}
\end{prop}
\begin{proof}
Recall that in a Krull--Schmidt category, an object is indecomposable if and only if its endomorphism ring is a local ring.  Part~\eqref{it:iota-indec} is immediate from the fact that $\iota$ is fully faithful.

For part~\eqref{it:pi-indec}, consider the object $\ir\pil(T'') \in \fD'$.  By Definition~\ref{defn:cotstruc}, there
 exists a (noncanonical) distinguished triangle $A \to \ir\pil(T'') \to B \to $ with $A \in \fD'_{\ge 0}[-1]$ and $B \in \fD'_{\le 0}$.  Consider the composition $\iota(A) \to \iota\ir\pil(T'') \to \pil(T'')$.  Let $\tilde T$ be the cone of this map, so that we have a distinguished triangle
\[
\iota(A) \to \pil(T'') \to \tilde T \to.
\]
It is straightforward to see from the definitions that $\tilde T$ lies in $\fC$.  Moreover, since 
$\Pi(\tilde T) \cong T''$ is indecomposable, $\tilde T$ must have a unique indecomposable summand that is not killed by $\Pi$.  Let $T$ denote this summand.  This is an indecomposable object in $\fC$ satisfying $\Pi(T) \cong T''$.

Before proving the uniqueness of $T$, let us consider part~\eqref{it:all-indec}.  Let $X \in \fC$ be an indecomposable object.  If $\Pi(X) = 0$, then the recollement formalism implies that $X \cong \iota(T')$, where $T' = \il(X) \cong \ir(X)$, so we are in case~\eqref{it:iota-indec}.  If $\Pi(X) \ne 0$, then by Lemma~\ref{lem:recolle-full}, it is indecomposable.  Let $T'' = \Pi(X)$.

The uniqueness in part~\eqref{it:pi-indec} and the remainder of part~\eqref{it:all-indec} both come down to the following assertion: \textit{If $X \in \fC$ is indecomposable and $\Pi(X) \cong T''$, then $X \cong T$.}  Let us prove this claim.  Choose an isomorphism $\theta: \Pi(X) \simto T''$.  By Lemma~\ref{lem:recolle-full}, the maps
\[
\Hom(X,T) \to \Hom(\Pi(X),T'') 
\qquad\text{and}\qquad
\Hom(T,X) \to \Hom(T'',\Pi(X))
\]
are both surjective.  Therefore, there exist maps $\phi: X \to T$ and $\psi: T \to X$ such that $\Pi(\phi) = \theta$ and $\Pi(\psi) = \theta^{-1}$.  Note that $\phi \circ \psi$ is an element of the local ring $\End(T)$ whose image in $\End(T'')$ is the identity map.  It follows that $\phi \circ \psi$ is invertible.  The same reasoning shows that $\psi \circ \phi$ is invertible, and hence that $\phi$ and $\psi$ are themselves isomorphisms, as desired.
\end{proof}

\begin{rmk}\label{rmk:indec-class}
Let $T$ and $T''$ be as in Proposition~\ref{prop:indec-class}\eqref{it:pi-indec}.  By adjunction, there are natural maps $\pil(T'') \to T \to \pir(T'')$.  The recollement formalism implies that the composition of these maps is equal to the canonical map $\pil(T'') \to \pir(T'')$: see~\cite[\S1.4.6(b)]{bbd}.
\end{rmk}

\subsection{Categories with a Tate twist}\label{ss:tate-twist}
\label{ss:tate}

For the remainder of this section, we will work in the ``graded'' setting: we always assume that our triangulated category $\fD$ is equipped with an autoequivalence 
\[
\lb 1\rb: \fD \to \fD
\]
such that for an object $X \in \fD$, we have
\[
X\lb 1 \rb \cong X
\qquad\text{if and only if}\qquad X = 0.
\]
This functor is called the \emph{shift-of-grading functor} (or sometimes the \emph{Tate twist}). For any object $X \in \fD$ and any $n \in \Z$, we write $X\lb n\rb$ for the object obtained by applying the shift-of-grading functor $n$ times.

Given a set of objects $\cX$ in a triangulated category $\fD$ equipped with a shift-of-grading functor $\lb 1\rb$, we define four subcategories of $\fD$ as follows:
\[
\begin{array}{l@{\qquad}l}
&
\text{\it\dots is defined to be the smallest full subcategory} \\
\text{\it Notation}  &\text{\it of $\fD$ containing $\cX$ and closed under \dots} \\
\hline
\lgen \cX \rgen_\ext & \text{extensions and $\lb {\pm1}\rb$} \\
\lgen \cX \rgen_\exto & \text{extensions, $\lb {\pm1}\rb$, and direct summands} \\
\lgen \cX \rgen_\tri & \text{$[\pm1]$, extensions and $\lb {\pm1}\rb$} \\
\lgen \cX \rgen_\trio & \text{$[\pm1]$, extensions and $\lb {\pm1}\rb$, and direct summands}
\end{array}
\]
The latter two are triangulated categories.  We obviously have $\lgen \cX \rgen_\ext \subset \lgen \cX \rgen_\exto$ and $\lgen \cX \rgen_\tri \subset \lgen \cX \rgen_\trio$; in some situations, these containments are equalities.

We will also use a graded variant of the ``$*$'' operation from~\cite[\S1.3.9]{bbd}.  Given $\cX_1, \cX_2 \subset \fD$, we define
\[
\cX_1 \uast \cX_2
\]
to be the full subcategory of $\fD$ consisting of objects $X$ such that there is a distinguished triangle
\[
A_1\lb n_1 \rb \to X \to A_2\lb n_2\rb \to
\qquad\text{with $A_1 \in \cX_1$, $A_2 \in \cX_2$, and $n_1, n_2 \in \Z$.}
\]
The proof of~\cite[Proposition~1.3.10]{bbd} shows that $\uast$ is associative.  For each $n\geq 0$, we can set 
\[
\cX^{\uast n} = \underbrace{\cX \uast \cdots \uast \cX}_{\text{$n$ factors}},
\] 
where the $0$-fold $\uast$-power of $\cX$ is understood to consist of just the zero object.
Observe that if we let $\widehat{\cX} = \{ X[i] \mid X \in \cX,\ i \in \Z \}$, then
\[
\lgen \cX\rgen_\ext = \bigcup_{n \ge 0} \cX^{\uast n}
\qquad\text{and}\qquad
\lgen \cX\rgen_\tri = \bigcup_{n \ge 0} \widehat{\cX}^{\uast n}.
\]

The following statement is the graded analogue of Proposition~\ref{prop:single-gen}.

\begin{prop}\label{prop:gr-single-gen}
Let $\bk$ be a field, and let $\fD$ be a $\bk$-linear triangulated category equipped with a shift-of-grading functor $\lb 1\rb: \fD \to \fD$.  Suppose there is an object $A$ such that $\fD = \lgen A \rgen_\tri$, and assume that
\[
\Hom(A, A\lb n\rb) \cong
\begin{cases}
\bk & \text{if $n = 0$,} \\
0 & \text{otherwise.}
\end{cases}
\]
Let
\[
\fD(\le 0) = \lgen A[i] : i \ge 0 \rgen_\exto
\qquad\text{and}\qquad
\fD(\ge 0) = \lgen A[i] : i \le 0 \rgen_\exto.
\]
Then:
\begin{enumerate}
\item If $\Hom(A,A\lb n\rb[i]) = 0$ for $i < 0$, then $(\fD(\le 0), \fD(\ge 0))$ is a $t$-structure on $\fD$.  In fact, it is the unique $t$-structure stable under
$\lb {\pm 1} \rb$ whose heart contains $A$.  Moreover, the heart is a finite-length abelian category, and the simple objects are of the form $A\lb n\rb$.\label{it:single-t}
\item If $\Hom(A,A\lb n\rb[i]) = 0$ for $i > 0$, then $(\fD(\ge 0), \fD(\le 0))$ is a co-$t$-structure on $\fD$.  In fact, it is the unique co-$t$-structure 
 stable under $\lb {\pm 1} \rb$ whose coheart contains $A$. Moreover, if $\fD$ is Karoubian, then the coheart is a Krull--Schmidt category, and the indecomposable objects are of the form $A\lb n\rb$.\label{it:single-cot}
\end{enumerate}
\end{prop}
\begin{proof}
\eqref{it:single-t}~The ungraded analogue of this claim is proved in~\cite[Lemma~3]{bez-quasi} (see~\cite[Proposition~A.1]{modular-LV} for a related argument), and that proof is easily adapted to the graded case.  More precisely, it is shown in {\it loc.~cit.} that the categories $\lgen A[i] : i \ge 0 \rgen_\ext$ and $\lgen A[i] : i \le 0 \rgen_\ext$ constitute a $t$-structure.  The subcategories defining a $t$-structure are automatically closed under direct summands (this follows from the existence of truncation functors), so in the setting of part~\eqref{it:single-t} of the proposition, we have $\lgen A[i] : i \ge 0 \rgen_\ext = \lgen A[i] : i \ge 0 \rgen_\exto$ and $\lgen A[i] : i \le 0 \rgen_\ext = \lgen A[i] : i \le 0 \rgen_\exto$.

\eqref{it:single-cot}~Axioms~\eqref{ax:thick} and~\eqref{ax:shift} from Definition~\ref{defn:cotstruc} are clear, and axiom~\eqref{ax:hom} follows from the assumption that $\Hom(A,A\lb n\rb[i]) = 0$ for $i > 0$.  Next, we claim that if $k > m$, then
\begin{equation}\label{eqn:star-sort}
A[k] \uast A[m] \subset A[m] \uast A[k].
\end{equation}
Indeed, the left-hand side consists of objects $X$ that fit into a distinguished triangle $A[k]\lb n_1\rb \to X \to A[m]\lb n_2\rb \to A[k+1]\lb n_1\rb$.  But since $k > m$, we have $k+1 > m$ as well, so our assumptions imply that the map $A[m]\lb n_2\rb \to A[k+1]\lb n_1\rb$ is zero, i.e., the triangle splits.  Thus, we have $X \cong A[k]\lb n_1\rb \oplus A[m]\lb n_2\rb \in A[m] \uast A[k]$.

For any $X \in \fD$, there exist integers $k_1, \ldots, k_j$ such that
\begin{equation}\label{eqn:star-gen}
X \in A[k_1] \uast \cdots \uast A[k_j].
\end{equation}
Using~\eqref{eqn:star-sort} repeatedly, we may assume that $k_1 \le k_2 \le \cdots \le k_j$.  Let $i$ be the unique subscript such that $k_i \le 0$ but $k_{i+1} > 0$ (here we permit $i = 0$ or $i = j$ if necessary).  Then
\begin{equation}\label{eqn:cotstruc-trunc}
X \in (A[k_1] \uast \cdots \uast A[k_i]) \uast (A[k_{i+1}] \uast \cdots \uast A[k_j]) \subset \fD(\ge 0) \uast (\fD(\le 0)[1]).
\end{equation}
We have proved axiom~\eqref{ax:dt}, and hence that $(\fD(\ge 0), \fD(\le 0))$ is a co-$t$-structure.

If there were another co-$t$-structure on $\fD$, say $(\fD'_{\ge 0}, \fD'_{\le 0})$, whose coheart contained $A$ and was stable under $\lb {\pm 1} \rb$, then we would clearly have $\fD(\ge 0) \subset \fD'_{\ge 0}$ and $\fD(\le 0) \subset \fD'_{\le 0}$.  This implies that the two co-$t$-structures coincide.  We have proved the uniqueness claim in the proposition.

From now on, we assume that $\fD$ is Karoubian.  Let $\fC = \fD(\ge 0) \cap \fD(\le 0)$ be the coheart of our co-$t$-structure.  We claim that $\fC$ is Krull--Schmidt.  It is enough to check condition~\eqref{it:ks-findim} from Lemma~\ref{lem:krull-schmidt}.  This condition follows easily from the observation that if $i \le 0 \le j$, then $\Hom(A[i], A[j]\lb n\rb)$ is finite-dimensional.

Note that $A\lb n\rb$ is indecomposable, because $\End(A) \cong \bk$ is a local ring.

Finally, let $X$ be a nonzero object in $\fC$.  It remains to show that $X$ is a direct sum of objects of the form $A\lb n\rb$.  Find an expression as in~\eqref{eqn:star-gen}, with $k_1 \le \cdots \le k_j$.  We proceed by induction on $j$. If $j = 1$, then $X \cong A[k_1]\lb n\rb$ for some $n \in \Z$.  Since $X \in \fC$, we must have $k_1 = 0$, and we are done.

From now on, suppose $j > 1$. Suppose first that $k_j > 0$.  Let $i$ be such that $k_i \le 0$ but $k_{i+1} > 0$, as in~\eqref{eqn:cotstruc-trunc}.  If $i = 0$, we would have $X \in \fD(\le 0)[1]$, a contradiction.  We therefore have $1 \le i < j$.  From~\eqref{eqn:cotstruc-trunc}, we get a distinguished triangle $X' \to X \to X'' \to$ with $X' \in \fD(\ge 0)$ and $X'' \in \fD(\le 0)[1]$.  The arrow $X \to X''$ must be $0$, so $X' \cong X \oplus X''[-1]$.  This tells us that $X''[-1] \in \fD(\ge 0)$, so $X''[-1]$ actually lies in the coheart.  The triangle $X''[-1] \to X' \to X \to$ shows that $X'$ lies in the coheart as well.  By induction, we know that $X' \in A[k_1] \uast \cdots \uast A[k_i]$ is a direct sum of objects of the form $A\lb n \rb$.  Since $X$ is a direct summand of $X'$, and since the coheart is Krull--Schmidt, our claim for $X$ holds.

Next, if $k_1 < 0$, we instead define $i$ to be such that $k_i < 0$ but $k_{i+1} \ge 0$, and then consider the decomposition~\eqref{eqn:cotstruc-trunc}.  The rest of the argument in this case is very similar to the preceding paragraph.

The remaining case is that in which $k_1 = \cdots = k_j = 0$.  In this case, we get a distinguished triangle $X' \to X \to A\lb m\rb \to$, where
\[
X' \in \underbrace{A \uast \cdots \uast A}_{\text{$j-1$ factors}}.
\]
We have $\Hom(A\lb m\rb, X'[1]) = 0$, so the triangle splits, and $X \cong X' \oplus A\lb m\rb$.  By induction, $X'$ is a direct sum of objects of the form $A\lb n\rb$, so we are done.
\end{proof}

\subsection{Pre-exceptional sets}
\label{ss:pre-excep}

We will now develop a vast generalization of Proposition~\ref{prop:gr-single-gen}.  Let $(S,\le)$ be a partially ordered set.  Assume that the partial order $\le$ admits a refinement to a total order $\le'$ such that $(S,\le')$ is isomorphic to a subset of $\Z_{\ge 0}$.  In particular, $S$ is either finite or countable, and 
\begin{equation}\label{eqn:lt-finite}
\text{the set $\{t \in S \mid t \le s \}$ is finite.}
\end{equation}

\begin{defn}\label{defn:excep}
Let $\fD$ be a triangulated category equipped with a Tate twist $\lb 1\rb$, and let $\{ \nabla_s \}_{s \in S}$ be a collection of objects in $\fD$ indexed by $S$. 
This collection is said to be a \emph{graded pre-exceptional set} if it satisfies the following axioms:
\begin{enumerate}
\item If $s \not\ge t$, then $\Hom(\nabla_s, \nabla_t \lb n \rb[i]) = 0$ for all $n, i \in \Z$.\label{it:exc-nabla-hom}
\item We have\label{it:exc-nabla-end}
\[
\Hom(\nabla_s, \nabla_s\lb n\rb) \cong
\begin{cases}
\bk & \text{if $n = 0$,} \\
0 & \text{otherwise.}
\end{cases}
\]
\item The collection $\{ \nabla_s\lb n\rb \}_{s \in S, n \in \Z}$ generates $\fD$ as a triangulated category.\label{it:exc-nabla-gen}
\end{enumerate}
Suppose $\{\nabla_s\}_{s \in S}$ is a graded pre-exceptional set.   Then:
\begin{enumerate}
\item[($4^+$)] If $\Hom(\nabla_s, \nabla_s\lb n\rb[i]) = 0$ for all $n \in \Z$ and all $i < 0$, then $\{\nabla_s\}$ is said to be a \emph{graded quasi-exceptional set}.
\item[($4^-$)] If $\Hom(\nabla_s, \nabla_s\lb n\rb[i]) = 0$ for all $n \in \Z$ and all $i > 0$, then $\{\nabla_s\}$ is said to be a \emph{graded co-quasi-exceptional set}.
\item[($4^\pm$)] If $\Hom(\nabla_s, \nabla_s\lb n\rb[i]) = 0$ for all $n \in \Z$ and all $i \ne 0$, then $\{\nabla_s\}$ is said to be a \emph{graded exceptional set}.
\end{enumerate}
\end{defn}

Exceptional and quasi-exceptional sets have been studied elsewhere in the literature, but we are not aware of any study of co-quasi-exceptional sets.  It turns out, however, that many of the basic lemmas about exceptional and quasi-exceptional sets can be proved using only the axioms for pre-exceptional sets, so they remain valid for co-quasi-exceptional sets as well.

\begin{rmk}
\begin{enumerate}
\item Of course, Definition~\ref{defn:excep} has an obvious ungraded analogue, as do the definitions and propositions below.  We leave it to the reader to formulate the precise statements.  
\item The definition of ``graded quasi-exceptional'' given above matches that in~\cite{bez-quasi}, but not that in~\cite{achar-perv}: the latter includes an extra condition that we will not impose.
\end{enumerate}
\end{rmk}

If $\fD$ is equipped with a graded pre-exceptional set $\{\nabla_s\}_{s \in S}$, then for any lower set $U \subset S$, we define
\[
\fD_U = \lgen \nabla_u : u \in U \rgen_\tri.
\]
(Recall that a subset $U \subset S$ is called a \emph{lower set} if $t \in U$ and $u \le t$ implies $u \in U$.)  For the special case $U = \{u \in S \mid u\le s\}$, we simply write $\fD_{\le s}$.  The categories $\fD_{< s}$ and $\fD_{\not\ge s}$ are defined similarly.

\begin{defn}\label{defn:dualizable}
Let $\{\nabla_s\}_{s \in S}$ be a graded pre-exceptional set.  It is said to be \emph{dualizable} if for each $s \in S$, there is an object $\Delta_s \in \fD$ and a morphism $\iota_s: \Delta_s \to \nabla_s$ such that the following two conditions hold:
\begin{enumerate}
\item The cone of $\iota_s: \Delta_s \to \nabla_s$ lies in $\fD_{<s}$.\label{it:dual-cone}
\item If $s > t$, then $\Hom(\Delta_s, \nabla_t\lb n\rb[i]) = 0$ for all $n, i \in \Z$.\label{it:dual-hom}
\end{enumerate}
\end{defn}

\begin{rmk}
According to~\cite[Proposition~3(b)]{bez-tilting}, every graded exceptional set is automatically dualizable.  (That statement assumes that $S$ is finite, but the same reasoning applies in our situation thanks to~\eqref{eqn:lt-finite}.)
\end{rmk}

\begin{lem}
Let $\{\nabla_s\}_{s \in S}$ be a dualizable graded pre-exceptional set.  For each $s \in S$, the pair $(\Delta_s, \iota_s: \Delta_s \to \nabla_s)$ is unique up to unique isomorphism.
\end{lem}
\begin{proof}
Suppose there is another pair $(\Delta'_s, \iota'_s: \Delta'_s \to \nabla_s)$ satisfying the conditions in Definition~\ref{defn:dualizable}.  Those conditions imply that if $X \in \fD_{<s}$, then $\Hom(\Delta_s, X) = \Hom(\Delta'_s, X) = 0$.

Let $K$ and $K'$ denote the cones of $\iota_s$ and $\iota'_s$, respectively.  By the observation above, we have
\[
\Hom(\Delta_s, K') = \Hom(\Delta_s, K'[-1]) = 0.
\]
By~\cite[Proposition~1.1.9]{bbd}, there is a unique map $f: \Delta_s \to \Delta'_s$ that makes the left-hand square in the diagram
\[
\begin{tikzcd}
\Delta_s \ar[d, "f"', dashed] \ar[r, "\iota_s"] &
  \nabla_s \ar[d, equal] \ar[r] &
  K \ar[r] & {} \\
\Delta'_s \ar[r, "\iota'_s"] &
  \nabla_s \ar[r] &
  K' \ar[r] & {}
\end{tikzcd}
\]
commute.  By swapping the roles of $\Delta_s$ and $\Delta'_s$, one can see that $f$ must be an isomorphism.
\end{proof} 

\begin{lem}\label{lem:pre-exc-basic}
Let $\{\nabla_s\}_{s \in S}$ be a dualizable graded pre-exceptional set, and let $\{\Delta_s\}_{s \in S}$ be the dual set.  Then:
\begin{enumerate}
\item If $s \ne t$, then $\Hom(\Delta_s, \nabla_t\lb n\rb[i]) = 0$ for all $n, i \in \Z$.\label{it:delta-nabla-orth}
\item If $s \not\le t$, then $\Hom(\Delta_s, \Delta_t \lb n \rb[i]) = 0$ for all $n, i \in \Z$.\label{it:exc-delta-hom}
\item For all $n, i \in \Z$, there are natural isomorphisms\label{it:exc-delta-end}
\[
\Hom(\Delta_s, \Delta_s\lb n\rb[i])
\cong \Hom(\Delta_s, \nabla_s\lb n\rb[i])
\cong \Hom(\nabla_s, \nabla_s\lb n\rb[i]).
\]
\item For any lower set $U \subset S$, we have $\fD_U = \lgen \Delta_u[i] : u \in U \rgen_\tri$.\label{it:exc-delta-gen}
\end{enumerate}
\end{lem}
\begin{proof}
Note that Definition~\ref{defn:dualizable}\eqref{it:dual-cone} implies that
\begin{equation}\label{eqn:delta-subcat}
\Delta_u \in \fD_{\le u}\qquad\text{for all $u \in S$.}
\end{equation}
Definition~\ref{defn:excep}\eqref{it:exc-nabla-hom} and Definition~\ref{defn:dualizable}\eqref{it:dual-hom} imply that
\begin{align}
\Hom(X,\nabla_u) &= 0 &&\text{for all $X \in \fD_{\not\ge u}$,} \label{eqn:nabla-homvan}
\\
\Hom(\Delta_u, X) &= 0 &&\text{for all $X \in \fD_{<u}$,}\label{eqn:delta-homvan-pre}
\end{align}
respectively.  Let us now prove the various parts of the present lemma.

\eqref{it:delta-nabla-orth} If $s > t$, this holds by Definition~\ref{defn:dualizable}\eqref{it:dual-hom}.  If $s \not\ge t$, it follows from~\eqref{eqn:delta-subcat} and~\eqref{eqn:nabla-homvan}. 

\eqref{it:exc-delta-hom}  By part~\eqref{it:delta-nabla-orth}, we can strengthen~\eqref{eqn:delta-homvan-pre} to
\begin{align}\label{eqn:delta-homvan}
\Hom(\Delta_u, X) &= 0 &&\text{for all $X \in \fD_{\not\ge u}$.}
\end{align}
The claim follows from this and~\eqref{eqn:delta-subcat}.

\eqref{it:exc-delta-end} Let $K$ be the cone of $\iota_s: \Delta_s \to \nabla_s$, and consider the long exact sequence
\begin{multline*}
\cdots \to \Hom(\Delta_s, K\lb n \rb [i-1]) 
\to \Hom(\Delta_s, \Delta_s\lb n \rb [i]) \\
\to \Hom(\Delta_s, \nabla_s\lb n \rb [i])
\to \Hom(\Delta_s, K\lb n \rb [i]) \to \cdots.
\end{multline*}
The first and last terms vanish by~\eqref{eqn:delta-homvan-pre}, so the middle two are naturally isomorphic.  The proof that $\Hom(\Delta_s, \nabla_s\lb n \rb [i]) \cong \Hom(\nabla_s, \nabla_s\lb n \rb [i])$ is similar.

\eqref{it:exc-delta-gen} Let $\fD'_U = \lgen \Delta_u : u \in U\rgen_\tri$.  It follows from~\eqref{eqn:delta-subcat} that $\fD'_U \subset \fD_U$.  To prove equality, let us first consider the case where $U$ is finite.  We proceed by induction on the number of elements in $U$.  Choose a maximal element $t \in U$, and let $V = U \smallsetminus \{t\}$.  Then $\fD'_V = \fD_V$ by induction, so $\fD_U$ is generated by $\fD'_V$ together with $\{\nabla_t\lb n\rb\}_{n \in \Z}$.  By Definition~\ref{defn:dualizable}\eqref{it:dual-cone}, it is also generated by $\fD'_V$ together with $\{\Delta_t\lb n\rb\}_{n \in \Z}$.  We conclude that $\fD'_U = \fD_U$.

If $U$ is infinite, any object $X \in \fD_U$ is still contained in a subcategory $\fD_{U'}$ for some finite subset $U' \subset U$, so the preceding paragraph tells us that $\fD'_U = \fD_U$ in this case as well.
\end{proof}

\begin{lem}\label{lem:recolle-exc}
Let $\{\nabla_s\}_{s \in S}$ be a dualizable graded pre-exceptional set, and let $U \subset S$ be a lower set.  Let $t$ be a maximal element of $U$, and let $U' = U \smallsetminus \{t\}$.  Let $\iota: \fD_{U'} \hookrightarrow \fD_U$ be the inclusion functor, and let $\Pi: \fD_U \to \fD_U/\fD_{U'}$ be the Verdier quotient functor.  Both of these functors admit left and right adjoints, and together, the six functors
\[
\begin{tikzcd}
\fD_{U'} \ar[r, "\iota" description] &
\fD_U \ar[r, "\Pi" description] \ar[l, bend right, "\il"'] \ar[l, bend left, "\ir" ] &
\fD_U/\fD_{U'} \ar[l, bend right, "\pil"'] \ar[l, bend left, "\pir" ]
\end{tikzcd}
\]
constitute a recollement diagram.
\end{lem}
\begin{proof}[Proof Sketch]
This result is proved in~\cite[Lemma~4]{bez-quasi}.  Here, we will briefly indicate the main steps of the argument.

Define $\fD_t = \lgen \Delta_t \rgen_\tri$ and $\fD^t = \lgen \nabla_t\rgen_\tri$.  The first step is to show that for any $X \in \fD_U$, there are distinguished triangles
\begin{equation}\label{eqn:recolle-dt}
\begin{aligned}
Y_1 \to &X \to Y_2 \to &  &\text{with $Y_1 \in \fD_t$, $Y_2 \in \fD_{U'}$,} \\
Z_1 \to &X \to Z_2 \to &  &\text{with $Z_1 \in \fD_{U'}$, $Z_2 \in \fD^t$.}
\end{aligned}
\end{equation}

Second, the functor $\Pi$ induces equivalences of categories
\[
\Pi|_{\fD_t}: \fD_t \simto \fD_U/\fD_{U'},
\qquad
\Pi|_{\fD^t}: \fD^t \simto \fD_U/\fD_{U'}.
\]
We define $\pil$ and $\pir$ to be their respective inverses (composed with the inclusion functor into $\fD_U$).

Finally, one shows (using~\cite[Proposition~1.1.9]{bbd}) that the distinguished triangles in~\eqref{eqn:recolle-dt} are in fact functorial.  There are canonical isomorphisms
\[
Y_1 \cong \pil(\Pi(X)), \qquad
Z_2 \cong \pir(\Pi(X)),
\]
and the maps $Y_1 \to X$ and $X \to Z_2$ are the adjunction maps that make $\pil$ and $\pir$ into the left and right adjoints of $\Pi$, respectively.  On the other hand, we define $\il(X) = Y_2$ and $\ir(X) = Z_1$.
\end{proof}

In the setting of Lemma~\ref{lem:recolle-exc}, it follows from Definition~\ref{defn:dualizable}\eqref{it:dual-cone} that
\begin{equation}\label{eqn:recolle-quot}
\Pi(\Delta_t) \cong \Pi(\nabla_t).
\end{equation}

%

\begin{prop}\label{prop:preexcep-struc}
Let $\{\nabla_s\}_{s \in S}$ be a dualizable graded pre-exceptional set, and let $\{\Delta_s\}_{s \in S}$ be the dual set.
\begin{enumerate}
\item If $\{\nabla_s\}_{s \in S}$ is a graded quasi-exceptional set, then the pair of subcategories\label{it:preexcep-t}
\[
\fD^{\le 0} = \lgen \Delta_s[i] : s \in S, i \ge 0 \rgen_\ext,
\qquad
\fD^{\ge 0} = \lgen \nabla_s[i] : s \in S, i \le 0 \rgen_\ext
\]
is a bounded $t$-structure on $\fD$.  Moreover, its heart $\fA$ is a finite-length abelian category.  For each $s \in S$, there is a unique simple object $L_s$ in $\fA$ that fits into a commutative diagram
\[
\begin{tikzcd}
\Delta_s \ar[rr, "\iota_s"] \ar[dr] && \nabla_s \\
& L_s \ar[ur]
\end{tikzcd}
\]
and every simple object is isomorphic to $L_s\lb n\rb$ for a unique pair $(s,n) \in S \times \Z$.
\item If $\{\nabla_s\}_{s \in S}$ is a graded co-quasi-exceptional set, then the pair of subcategories\label{it:preexcep-cot}
\[
\fD_{\ge 0} = \lgen \Delta_s[i] : s \in S, i \le 0 \rgen_\exto,
\qquad
\fD_{\le 0} = \lgen \nabla_s[i] : s \in S, i \ge 0 \rgen_\exto
\]
is a bounded co-$t$-structure.  Moreover, if $\fD$ is Karoubian, then its coheart $\fC$ is a Krull--Schmidt additive category.  For each $s \in S$, there is a unique indecomposable object $T_s$ in $\fC$ that fits into a commutative diagram
\[
\begin{tikzcd}
\Delta_s \ar[rr, "\iota_s"] \ar[dr] && \nabla_s \\
& T_s \ar[ur]
\end{tikzcd}
\]
and every indecomposable object is isomorphic to $T_s\lb n\rb$ for a unique pair $(s,n) \in S \times \Z$.
\end{enumerate}
\end{prop}
\begin{proof}
\eqref{it:preexcep-t}~This is proved in~\cite[Propositions~1 and~2]{bez-quasi}.  A brief outline of the proof is as follows: one uses recollement to build up the $t$-structure from the case where $S$ is a singleton, and in that case, the $t$-structure comes from Proposition~\ref{prop:gr-single-gen}.

\eqref{it:preexcep-cot}~We follow the main idea of the $t$-structure case.  Assume first that $S$ is finite.  We proceed by induction on the size of $S$.  If $S$ is empty, then $\fD = 0$, and there is nothing to prove.  Otherwise, let $t$ be a maximal element of $S$, and let $S' = S \smallsetminus \{t\}$.  Form the recollement diagram for $\fD_{S'}$, $\fD_S = \fD$, and $\fD/\fD_{S'}$ as in Lemma~\ref{lem:recolle-exc}.  Let $A_t = \Pi(\Delta_t) = \Pi(\nabla_t)$ (see~\eqref{eqn:recolle-quot}), and observe that
\[
\fD/\fD_{S'} = \lgen A_t[i] : i \in \Z \rgen_\ext.
\]
The proof of Lemma~\ref{lem:recolle-exc} shows that
\begin{equation}\label{eqn:recolle-ind-step}
\pil(A_t) \cong \Delta_t
\qquad\text{and}\qquad
\pir(A_t) \cong \nabla_t.
\end{equation}
We have
\begin{multline*}
\Hom(A_t,A_t\lb n\rb[i]) \cong \Hom(\nabla_t, \pir\Pi(\nabla_t)\lb n\rb[i]) \\
\cong \Hom(\nabla_t,\nabla_t\lb n\rb[i]) 
\cong
\begin{cases}
\bk & \text{if $i = n = 0$,} \\
0 & \text{if $i > 0$, or if $i = 0$ and $n \ne 0$.}
\end{cases}
\end{multline*}
By Proposition~\ref{prop:gr-single-gen}, $\fD/\fD_{S'}$ admits a unique co-$t$-structure whose coheart contains $A_t$.  Combining the description from that proposition with~\eqref{eqn:recolle-ind-step}, we obtain
\begin{equation}\label{eqn:recolle-ind1}
\begin{aligned}
\pil((\fD/\fD_{S'})_{\ge 0}) &\subset \lgen \Delta_t[i] : i \le 0 \rgen_\exto, \\
\pir((\fD/\fD_{S'})_{\le 0}) &\subset \lgen \nabla_t[i] : i \ge 0 \rgen_\exto.
\end{aligned}
\end{equation}
(It follows from~\cite[(1.4.3.5)]{bbd} that these containments are actually equalities, but we will not need this claim.)

On the other hand, by induction, $\fD_{S'}$ has a co-$t$-structure given by
\begin{equation}\label{eqn:recolle-ind2}
\fD_{S',\ge 0} = \lgen \Delta_s[i] : s \in S', i \le 0 \rgen_\exto,
\qquad
\fD_{S',\le 0} = \lgen \nabla_s[i] : s \in S', i \ge 0 \rgen_\exto.
\end{equation}
We can then apply Proposition~\ref{prop:recolle-cotstruc} to obtain a co-$t$-structure on $\fD$ given by
\begin{align*}
\fD_{\ge 0} &= \{ X \mid \il(X) \in \fD_{S',\ge 0}, \Pi(X) \in (\fD/\fD_{S'})_{\ge 0} \}, \\
\fD_{\le 0} &= \{ X \mid \ir(X) \in \fD_{S',\le 0}, \Pi(X) \in (\fD/\fD_{S'})_{\le 0} \}.
\end{align*}
%
For $X \in \fD_{\ge 0}$, consider the distinguished triangle $\pil\Pi(X) \to X \to \iota\il(X) \to$.  
%
%
In view of~\eqref{eqn:recolle-ind1} and~\eqref{eqn:recolle-ind2}, we see that $X \in \lgen \Delta_s[i]: s \in S, i \le 0\rgen_\exto$.  Thus, we have shown the first containment below:
\begin{equation}\label{eqn:recolle-contain}
\fD_{\ge 0} \subset \lgen \Delta_s[i] : s \in S, i \le 0 \rgen_\exto,
\qquad
\fD_{\le 0} \subset \lgen \nabla_s[i] : s \in S, i \ge 0 \rgen_\exto.
\end{equation}
The second holds by a similar argument.  One can check using Lemma~\ref{lem:pre-exc-basic} that if $X \in \lgen \Delta_s[i] : s \in S, i \le 0 \rgen_\exto$ and $Y \in \lgen \nabla_s[i] : s \in S, i \ge 0 \rgen_\exto$, then $\Hom(X,Y[1]) = 0$.  From this observation, it can be deduced that both containments in~\eqref{eqn:recolle-contain} are equalities.  This completes the construction of the co-$t$-structure in the case where $S$ is finite.  We denotes its coheart by $\fC$.

From now on, we assume that $\fD$ is Karoubian.  We claim that $\fC$ is Krull--Schmidt.  It is enough to check condition~\eqref{it:ks-findim} from Lemma~\ref{lem:krull-schmidt}.  This condition follows easily from the observation that if $i \le 0 \le j$, then $\Hom(\Delta_s[i], \nabla_{s'}[j]\lb n\rb)$ is finite-dimensional.  The classification and description of indecomposable objects in $\fC$ follows from Proposition~\ref{prop:indec-class} and Remark~\ref{rmk:indec-class}, and induction on $S$.

Finally, suppose $S$ is infinite.  Any finite collection of objects is contained in some subcategory $\fD_U$ where $U \subset S$ is a finite lower set.  The axioms for a co-$t$-structure hold for $\fD$ because they hold for each such $\fD_U$.
\end{proof}

\begin{rmk}\label{rmk:silting-gen}
If $\fC$ is the coheart of a co-$t$-structure on $\fD$ obtained from a graded co-quasi-exceptional set as in Proposition~\ref{prop:preexcep-struc}\eqref{it:preexcep-cot}, then we claim that
\[
\fD = \lgen \fC \rgen_\tri.
\]
To see this, it is enough to check that the right-hand side contains all $\nabla_s$, and this is easily seen by induction on $S$ using the recollement formalism.  Note that this is stronger than property~\eqref{it:silting-gen} in Definition~\ref{defn:silting-cat}.
\end{rmk}

We now consider a special case where the co-$t$-structure and $t$-structure defined in the preceding proposition are 
highly compatible. 
\begin{cor}\label{prop:silt-tilt}
Let $\{\nabla_s\}_{s \in S}$ be a dualizable graded exceptional set, with dual set $\{\Delta_s\}_{s \in S}$.  Let $\fA$ be the heart of the corresponding $t$-structure, and let $\fC$ be the coheart of the corresponding co-$t$-structure.  Assume furthermore that $\nabla_s$ and $\Delta_s$ belong to $\fA$.
Then $\fA$ is a graded highest-weight category, and $\fC$ is the category of tilting objects in $\fA$.
\end{cor}
\begin{proof}
The fact that $\fA$ is a graded highest weight category follows from \cite[\S 3.5]{mr}. Moreover, by
Proposition~\ref{prop:preexcep-struc}\eqref{it:preexcep-cot}, it can be immediately deduced that the tilting objects of $\fA$ must reside in the
coheart. This forces the co-$t$-structure to coincide with the unique co-$t$-structure characterized by the preceding proposition. 
\end{proof}

\section{Cotangent bundles of partial flag varieties}
\label{sec:cotangent}

Let $G$ be a connected reductive group over an algebraically closed field $\bk$ of characteristic $p \ge 0$.  Assume that $p$ is ``pretty good'' for $G$, in the sense of~\cite[Definition~2.11]{herpel}.  This condition implies the following additional conditions:
\begin{enumerate}
\item \cite[Lemma~2.3]{modular-LV} There exists a separable central isogeny $\widetilde{G} \to G$, where $\widetilde{G}$ has a simply connected derived subgroup.
\item \cite[Lemma~2.12]{herpel} The characteristic of $\bk$ is good for $G$.
\item \cite[Proposition~12]{mt1} There exists a nondegenerate $G$-invariant bilinear form on $\fg$.
\end{enumerate}

Fix a Borel subgroup $B \subset G$ and a maximal torus $T$. Let $W = N_G(T)/T$ be the Weyl group, and let $\bX$ be the character lattice of $T$. Let $\Phi \subset \bX$ be the root system of $(G,T)$, and let $\Phi^+ \subset \Phi$ be the set of positive roots, chosen so that $B$ corresponds to the \emph{negative} roots.  Let $S$ be the set of simple reflections in $W$, and for $s \in S$, let $\alpha_s \in \Phi^+$ be the corresponding simple root.  Let $\bXp \subset \bX$ be the set of dominant weights corresponding to $\Phi^+$.

Let $\Waff = W \ltimes \bX$ be the (extended) affine Weyl group.  For $\lambda \in \bX$, let $t_\lambda = 1 \ltimes \lambda$ be the corresponding element of $\Waff$.  We will use additive notation for $\bX$ and multiplicative notation for $\Waff$, so that $t_{\lambda+\mu} = t_\lambda t_\mu$.  Recall that although $\Waff$ is not a Coxeter group in general, it shares many features with Coxeter groups.  In particular, it makes sense to speak of the \emph{length} of an element $w \in \Waff$, denoted by $\ell(w)$.  There is also a \emph{Bruhat order} on $\Waff$, denoted by $\le_\Bru$.  For each $\lambda \in \bX$, let
\[
w_\lambda = \text{the unique element of minimal length in the coset $Wt_\lambda$.}
\]
We define an order $\le$ on $\bX$ by
\[
\lambda \le \mu \qquad\text{if and only if}\qquad
w_\lambda \le_\Bru w_\mu.
\]

Next, for each subset $I \subset S$, we set
\begin{align*}
\bXp_I &= \{ \lambda \in \bX \mid \text{$\la \alpha_s^\vee, \lambda \ra \ge 0$ for all $s \in I$} \}, \\
\bXpp_I &= \{ \lambda \in \bX \mid \text{$\la \alpha_s^\vee, \lambda \ra > 0$ for all $s \in I$} \}.
\end{align*}
Let $P_I \subset G$ be the parabolic subgroup containing $B$ and corresponding to $I$, and let $U_I$ be its unipotent radical.  Let $\fn_I$ be the Lie algebra of $U_I$, and let 
\[
\tcN_I = G \times^{P_I} \fn_I.
\]
Note that $P_\varnothing = B$.  When $I = \varnothing$, we often omit the subscript: we write $U$ for the unipotent radical of $B$, and $\fn$ for its Lie algebra, and we denote
\[
\tcN = G \times^B \fn.
\]
Let $\Gm$ act on $\fn_I$ by $z \cdot x = z^{-2}x$. This action commutes with the action of $P_I$ and 
 induces a positive even grading on the algebra 
$\Sym(\fn_I^*)$ which we call the \emph{cohomological grading}. This allows us to form the bounded derived 
category 
\[
\Db\Coh^{P_I\times \Gm}(\fn_I) = \Db\Sym(\fn_I^*)\text{-mod}_{P_I\times \Gm},
\]
where $\Sym(\fn_I^*)\text{-mod}_{P_I\times \Gm}$ denotes the category of finitely generated $P_I$-equivariant graded modules 
over the algebra $\Sym(\fn_I^*)$. We can equivalently consider the category $\Db\Coh^{G\times \Gm}(\tcN_I)$
by recalling the equivalence 
\begin{equation}\label{eqn:restric-equiv}
j_I^*: \Coh^{G\times \Gm}(\tcN_I) \xrightarrow{\sim} \Coh^{P_I\times \Gm}(\fn_I).
\end{equation}
induced by pulling back along the inclusion $j_I: \fn_I \hookrightarrow \tcN_I$.

For any $m \in \Z$, let $\bk_m$ be the $1$-dimensional $\Gm$-representation with the action given by $z \cdot x = z^mx$.  Define an autoequivalence
\[
\la 1 \ra : \Db\Coh^{G\times \Gm}(\tcN_I) \to \Db\Coh^{G\times \Gm}(\tcN_I) 
\qquad\text{by}\qquad
\cF\la 1\ra = \cF \otimes \bk_{-1}.
\]
We will also work with the autoequivalence 
\[
\{1\} := \la -1\ra[1] : \Db\Coh^{G\times \Gm}(\tcN_I) \to \Db\Coh^{G\times \Gm}(\tcN_I).
\]
It is easy to see that both $\la 1 \ra$ and $\{1\}$ are Tate twists on $\Db\Coh^{G\times \Gm}(\tcN_I)$.
\begin{rmk}[Tate twist conventions]\label{rmk:twist-conventions}
It is important to note that the $\la m \ra$ employed here is \emph{opposite} to the convention in 
\cite[\S9.1]{achar-riche}, but consistent with \cite{achar-vogan}, \cite{achar-hardesty} and \cite{achar-perv}.  (More precisely, the Tate twist in \cite[\S 2.2]{achar-perv} is opposite to our convention, but the action on $\fn$ employed there is also opposite to ours, so statements from~\cite{achar-perv} can be used here unmodified.)
\end{rmk}

For $\lambda \in \bXpp_I$, let $\tnabla_{I,\lambda}$ and $\tDelta_{I,\lambda}$ be the objects defined in \cite[\S9.5]{achar-riche}. (For $I = \varnothing$, the definition of these objects goes back to~\cite{bez-tilting}; see also~\cite{arider,mr}.)  The following proposition gives a key property of these objects.

\begin{prop}[{\cite{acr,achar-riche,arider,bez-tilting,mr}}]\label{prop:cotangent-exc}
The collection $\{\tnabla_{I,\lambda}\}_{\lambda \in \bXpp_I}$ is a graded exceptional set with respect to both $\la 1\ra$ and $\{ 1\}$, and $\{\tDelta_{I, \lambda}\}_{\lambda \in \bXpp_I}$ is its dual set.
\end{prop}

For $\la 1\ra$, this statement appears in~\cite[\S9.5]{achar-riche}; for $\{1\}$, see~\cite[Lemma~3.1]{acr}.\footnote{In~\cite[\S9.5]{achar-riche}, it is assumed that $G$ has simply connected derived subgroup, but this assumption can be dropped by the reasoning explained in~\cite[\S4.1]{modular-LV}. The proof of~\cite[Lemma~3.1]{acr} then also applies in this generality.}  In the special case where $I = \varnothing$, the $\la 1\ra$ part of this statement was proved in~\cite{arider, mr}; the essential ideas go back to~\cite{bez-tilting}, which treats the case $\bk = \C$.

In view of Proposition~\ref{prop:cotangent-exc}, we may consider the following notions:
\begin{enumerate}
\item The \emph{exotic $t$-structure}, obtained by applying Proposition~\ref{prop:preexcep-struc}\eqref{it:preexcep-t} with respect to $\la 1\ra$.  This $t$-structure has been extensively studied in the papers mentioned above.
\item The \emph{representation-theoretic $t$-structure}, obtained by applying Proposition \ref{prop:preexcep-struc}\eqref{it:preexcep-t} with respect to $\{ 1\}$. This $t$-structure was implicitly used in~\cite{achar-riche}, and explicitly studied in~\cite{acr}.
\item The \emph{supportive co-$t$-structure}, obtained by applying Proposition~\ref{prop:preexcep-struc}\eqref{it:preexcep-cot} with respect to $\{ 1\}$. This co-$t$-structure is one the main objects of study of the present paper.
\item One may also apply Proposition~\ref{prop:preexcep-struc}\eqref{it:preexcep-cot} with respect to $\la 1\ra$, but the resulting co-$t$-structure does not appear to be useful for the goals of this paper, and will not be used.
\end{enumerate}

The simple objects in the heart of the exotic $t$-structure are indexed (up to Tate twist) by $\bXpp_I$; we set  
\[
\tfL_{I,\lambda} = 
\text{the simple exotic object labeled by $\lambda \in \bXpp_I$}.
\]

The coheart of the supportive co-$t$-structure (so named because of its role in the study of support varieties) is denoted by $\cS^{G \times \Gm}(\tcN_I)$.  We emphasize that this is the only co-$t$-structure on $\Db\Coh^{G\times \Gm}(\tcN_I)$ we will consider. Thus, objects of $\cS^{G \times \Gm}(\tcN_I)$ may simply be called ``silting objects on $\tcN_I$.''  The indecomposable silting objects are also indexed by $\bXpp_I$; we set 
\begin{equation}\label{eqn:para-silting}
\tfE_{I,\lambda} = \text{the indecomposable silting object labeled by $\lambda \in \bXpp_I$}.
\end{equation}
When $I = \varnothing$, the subscript $I$ will often be omitted from the notation.

The representation-theoretic $t$-structure is a highest-weight category  (see~\cite[\S3.D]{acr}), and so the following lemma is a consequence of Corollary~\ref{prop:silt-tilt}. 

\begin{lem}\label{lem:tcn-silt-tilt}
The category $\cS^{G \times \Gm}(\tcN_I)$ is the category of tilting objects in the heart of the representation-theoretic $t$-structure on $\Db\Coh^{G\times \Gm}(\tcN_I)$.
\end{lem}

In the case $I = \varnothing$, the objects $\tfE_\lambda$ were introduced in~\cite{achar-hardesty-riche} in a different framework: they were obtained from the machinery of ``parity objects,'' rather than from co-$t$-structures.  The following fact, explained in~\cite[Remark~3.8]{acr}, says that these two approaches yield the same objects.

\begin{lem}\label{lem:silting-parity}
The objects $\tfE_\lambda \in \Db\Coh^{G\times \Gm}(\tcN)$ are precisely the parity exotic objects of~\cite[\S3.4]{achar-hardesty-riche}.
\end{lem}

(As explained in~\cite[Remark~3.8]{acr}, an analogous statement holds for $\tfE_{I,\lambda}$ for any $I$, but we will not need this more general statement.)
Lemma~\ref{lem:silting-parity} makes a number of results from~\cite{achar-hardesty-riche} available to us.

\section{The nilpotent cone}
\label{sec:nilpotent}

Let $\cN$ be the nilpotent cone of $G$ and let
$
\pi: \tcN \rightarrow \cN
$
be the Springer resolution. 
Each $\lambda \in \bX$ gives rise to an equivariant line bundle $\cO_{\tcN}(\lambda)$ on $\tcN$ and 
an \emph{Andersen--Jantzen sheaf}
\[
\ocA_{\lambda} = \pi_*\cO_{\tcN}(\lambda).
\]
(Note that $\pi_*: \Db\Coh^{G\times \Gm}(\tcN) \rightarrow \Db\Coh^{G\times \Gm}(\cN)$ denotes the derived push-forward.)

For any $\lambda \in \bX$, let 
\[
\delta_\lambda = \min \{ \ell(v) \mid \text{$v \in W$ and $v\lambda \in \bXp$} \},
\qquad
\delta^*_\lambda = \delta_{w_0\lambda},
\]
and $\dom \lambda = W\lambda\cap \bX^+$.
Finally, for $\lambda \in \bXp$, we set
\[
\onabla_\lambda = \ocA_\lambda\la -\delta^*_\lambda \ra,
\qquad
\oDelta_\lambda = \ocA_{w_0\lambda}\la \delta^*_\lambda \ra, \quad
\fL_\lambda = \pi_*\tfL_{w_0\lambda}.
\]

\begin{lem}\label{lem:exotic-push}
For any $\lambda \in \bX$, we have
\[
\pi_*\tnabla_\lambda \cong \onabla_{\dom \lambda}\la \delta^*_\lambda\ra,
\quad
\pi_*\tDelta_{\lambda} \cong \oDelta_{\dom \lambda}\la -\delta^*_\lambda\ra,
\quad 
\pi_*\tfL_\lambda = \begin{cases}
\fL_{w_0\lambda} & \text{if $\lambda \in -\bX^+$} \\
0 & \text{otherwise}.
\end{cases}
\]
\end{lem}
\begin{proof}
This is a consequence of \cite[Proposition~2.6]{achar-vogan} (see Remark~\ref{rmk:twist-conventions} on 
how our Tate twist compares with the one in \cite{achar-vogan}). 
\end{proof}

\begin{lem}\label{lem:n-hom-finite}
For any $\cF, \cG \in \Db\Coh^{G \times \Gm}(\tcN_I)$ or $\Db\Coh^{G \times \Gm}(\cN)$, the space
\[
\bigoplus_{k \in \Z} \Hom(\cF,\cG\la k\ra) 
\]
is finite-dimensional.
\end{lem}
\begin{proof}
It is enough to prove this when $\cF$ and $\cG$ are shifts of objects belonging to some set that generates the given triangulated category (under Tate twist).  For $\tcN_I$, we may take $\cF = \tDelta_{I,\lambda}[n]$ and $\cG = \tnabla_{I,\mu}[m]$ (cf.~\cite[Proposition~4.4]{achar-hardesty-riche}). Since these objects come from an exceptional set, Lemma~\ref{lem:pre-exc-basic} tells us that the sum $\bigoplus_k \Hom(\cF,\cG\la k\ra)$ is at most $1$-dimensional.

For $\cN$, take $\cF = \oDelta_\lambda[n]$ and $\cG = \onabla_\mu[m]$.  According to~\cite[Proposition~6.1]{achar-perv}, these objects come from a quasi-exceptional set, so if $\lambda \ne \mu$, then Lemma~\ref{lem:pre-exc-basic} says that our $\Hom$-space vanishes.  On the other hand, if $\lambda = \mu$, we have
\begin{multline*}
\Hom(\oDelta_\lambda[m], \onabla_\lambda[n]\la k\ra)
\cong \Hom(\pi_*\tDelta_{w_0\lambda}, \pi_*\tnabla_{w_0\lambda}[n-m]\la k\ra) \\
\cong \Hom(\pi^*\pi_*\tDelta_{w_0\lambda}, \tnabla_{w_0\lambda}[n-m]\la k\ra).
\end{multline*}
Since $\pi$ is not smooth, the derived coherent pullback functor $\pi^*$ takes values in $D^-\Coh^{G\times\Gm}(\tcN)$, and not in the bounded derived category.  But of course $\tnabla_{w_0\lambda}[n-m]\la k\ra$ is a bounded complex, so there is some integer $N$ such that
\[
\Hom(\oDelta_\lambda[m], \onabla_\lambda[n]\la k\ra) \cong \Hom(\tau^{\ge N} \pi^*\pi_*\tDelta_{w_0\lambda}, \tnabla_{w_0\lambda}[n-m]\la k\ra).
\]
The right-hand side is now a $\Hom$-group in $\Db\Coh^{G \times \Gm}(\tcN)$.  The direct sum of these over all $k$ is finite-dimensional by the previous paragraph.
\end{proof}

\begin{lem}\label{lem:end-onabla}
If $\Hom(\onabla_\lambda, \onabla_\lambda[n]\la k\ra) \ne 0$, then one of the following holds:
\begin{enumerate}
\item $n = k = 0$, or
\item $n > 0$ and $k \le -2n$.
\end{enumerate}
\end{lem}

Before proving this lemma, we recall some facts about $\Db\Coh^{G \times \Gm}(\cN)$ from~\cite{achar-perv}.  As noted above,~\cite[Proposition~6.1]{achar-perv} says that the set $\{\onabla_\lambda\}_{\lambda \in \bXp}$ is a dualizable graded quasi-exceptional with respect to $\la 1\ra$.  As in~\S\ref{ss:pre-excep}, we can consider the following subcategories of $\Db\Coh^{G \times \Gm}(\cN)$, defined with respect to $\la 1 \ra$:
\[
\fD_{\le \lambda} = \lgen \onabla_\mu : \text{$\mu \in \bXp$, $\mu \le \lambda$} \rgen_\tri, \qquad
\fD_{< \lambda} = \lgen \onabla_\mu : \text{$\mu \in \bXp$, $\mu < \lambda$} \rgen_\tri.
\]
Apply Lemma~\ref{lem:recolle-exc} to obtain a recollement diagram.  We denote the functors involving the quotient $\fD_{\le \lambda}/\fD_{< \lambda}$ as follows:
\[
\begin{tikzcd}[column sep=large]
\fD_{\le \lambda} \ar[r, "\Pi_\lambda" description]  &
\fD_{\le \lambda}/\fD_{<\lambda}. \ar[l, bend right, "\pil_\lambda"'] \ar[l, bend left, "\pir_\lambda" ]
\end{tikzcd}
\]
It follows from~\cite[Lemma~5.3]{achar-perv} that for all $\nu \in \bX$, we have
\[
\ocA_\nu \in \fD_{\le \dom(\nu)}
\qquad\text{and}\qquad
\Pi_{\dom(\nu)}(\ocA_\nu) \cong \Pi_{\dom(\nu)}(\ocA_{\dom(\nu)}\la -2\delta_\nu\ra).
\]
For $\lambda \in \bXp$ and $\nu \in \bXp$, we thus have
\begin{equation}\label{eqn:pi-aj}
\Pi_\lambda(\ocA_\nu) \cong
\begin{cases}
\Pi_\lambda(\ocA_\lambda\la -2\delta_\nu \ra) & \text{if $\nu \in W\lambda$,} \\
0 & \text{if $\dom(\nu) \le \lambda$ and $\nu \notin W\lambda$.}
\end{cases}
\end{equation}
We are now ready to prove Lemma~\ref{lem:end-onabla}.

\begin{proof}[Proof of Lemma~\ref{lem:end-onabla}]
Let $\weyl(\lambda)$ be the Weyl module for $G$ with highest weight $\lambda$, and consider the coherent sheaf $\cO_\cN \otimes \weyl(\lambda) \in \Coh^{G \times \Gm}(\cN)$.  According to~\cite[Lemma~5.4]{achar-perv}, we have
\begin{equation}\label{eqn:free-aj}
\cO_\cN \otimes \weyl(\lambda) \in \ocA_{\nu_1} * \cdots * \ocA_{\nu_k} * \ocA_\lambda
\end{equation}
for some weights $\nu_1, \ldots, \nu_k \in \bX$ that satisfy $\dom(\nu_i) \le \lambda$ and $\nu_i \ne \lambda$ for all $i$.  In particular, this shows that $\cO_\cN \otimes \weyl(\lambda)$ lies in $\fD_{\le \lambda}$, so it makes sense to apply $\Pi_\lambda$ to it. Let
\[
\cG = \Pi_\lambda(\cO_\cN \otimes \weyl(\lambda))\la -\delta^*_\lambda\ra.
\]
Combining~\eqref{eqn:free-aj} with~\eqref{eqn:pi-aj}, we find that $\cG$ belongs to a category of the form $\Pi_\lambda(\onabla_\lambda)\la -2r_1\ra * \Pi_\lambda(\onabla_\lambda)\la -2r_2\ra * \cdots * \Pi_\lambda(\onabla_\lambda)\la -2r_k\ra * \Pi_\lambda(\onabla_\lambda)$, where $r_1, \ldots, r_k$ are some positive integers.  Thus, there is a distinguished triangle
\[
\cG' \to \cG \to \Pi_\lambda(\onabla_\lambda)
\]
with
\[
\cG' \in \Pi_\lambda(\onabla_\lambda)\la -2r_1\ra * \Pi_\lambda(\onabla_\lambda)\la -2r_2\ra * \cdots * \Pi_\lambda(\onabla_\lambda)\la -2r_k\ra.
\]

The proof of Lemma~\ref{lem:recolle-exc} shows that $\pir_\lambda(\Pi_\lambda(\onabla_\lambda)) \cong \onabla_\lambda$.  We therefore have
\begin{multline*}
\Hom(\cO_\cN \otimes \weyl(\lambda)\la -\delta^*_\lambda\ra, \onabla_\lambda[n]\la k\ra) \\
\cong \Hom(\cO_\cN \otimes \weyl(\lambda)\la -\delta^*_\lambda\ra, \pir_\lambda(\Pi_\lambda(\onabla_\lambda))[n]\la k\ra) 
\cong \Hom(\cG, \Pi_\lambda(\onabla_\lambda)[n]\la k\ra) \\
\cong \Hom(\pir_\lambda(\cG), \pir_\lambda\Pi_\lambda(\onabla_\lambda)[n]\la k\ra)
\cong \Hom(\pir_\lambda(\cG), \onabla_\lambda[n]\la k\ra),
\end{multline*}
where the penultimate step uses the fact that $\pir_\lambda$ is fully faithful (\cite[(1.4.3.5)]{bbd}).

Let $\cF = \pir_\lambda(\cG)$.  Combining the calculation above with~\cite[Lemma~5.5(2)]{achar-perv}, we have 
\begin{equation}\label{eqn:hom-f-nabla}
\Hom(\cF, \onabla_\lambda[n]\la k\ra)
\cong
\begin{cases}
\bk & \text{if $n = k = 0$,} \\
0 & \text{otherwise.}
\end{cases}
\end{equation}

Next, let $\cF' = \pir_\lambda(\cG')$, so that we have a distinguished triangle
\[
\cF' \to \cF \to \onabla_\lambda \to
\]
where
\begin{equation}\label{eqn:fprime-filt}
\cF' \in \onabla_\lambda\la -2r_1\ra * \onabla_\lambda\la -2r_2\ra * \cdots * \onabla_\lambda\la -2r_k\ra.
\end{equation}
This distinguished triangle gives rise to a long exact sequence
\begin{multline}\label{eqn:nabla-les}
\cdots \to
\Hom(\cF', \onabla_\lambda[n-1]\la k\ra) \to 
\Hom(\onabla_\lambda, \onabla_\lambda[n]\la k\ra) \\ \to 
\Hom(\cF, \onabla_\lambda[n]\la k\ra) \to \cdots
\end{multline}

Let us first take $n$ to be the smallest integer such that $\Hom(\onabla_\lambda, \onabla_\lambda[n]\la k\ra) \ne 0$ for some $k \in \Z$.  In view of~\eqref{eqn:fprime-filt}, the first term in~\eqref{eqn:nabla-les} vanishes.  The second term is assumed to be nonzero, so the third term must be as well.  But by~\eqref{eqn:hom-f-nabla}, this implies that $n = k = 0$.  In particular, we have $\Hom(\onabla_\lambda, \onabla_\lambda[n]\la k\ra) = 0$ if $n < 0$.

Assume henceforth that $n \ge 0$.  We will prove by induction on $n$ that if $\Hom(\onabla_\lambda, \onabla_\lambda[n]\la k\ra) \ne 0$, then $k \le -2n$.  The case $n = 0$ has already been covered by the previous paragraph.  Suppose now that $n > 0$, and let $k$ be the largest integer such that $\Hom(\onabla_\lambda, \onabla_\lambda[n]\la k\ra) \ne 0$.  (Such a $k$ exists by Lemma~\ref{lem:n-hom-finite}.) The third term of~\eqref{eqn:nabla-les} vanishes (by~\eqref{eqn:hom-f-nabla}), so in order for the second term to be nonzero, the first term must also by nonzero.  By induction and~\eqref{eqn:fprime-filt}, we must have
\[
k + 2r_i \le -2(n-1)
\]
for some $i$.  That is, $k \le -2n+2(1-r_i)$.  Since $1 - r_i \le 0$, we conclude that $k \le -2n$.
\end{proof}

We can now prove a complement to Lemma~\ref{lem:n-hom-finite}.

\begin{lem}\label{lem:n-hom-finite-2}
For any $\cF, \cG \in \Db\Coh^{G \times \Gm}(\tcN_I)$ or $\Db\Coh^{G \times \Gm}(\cN)$, the space
\[
\bigoplus_{k \in \Z} \Hom(\cF,\cG\{ k\}) 
\]
is finite-dimensional.
\end{lem}
\begin{proof}
For $\tcN_I$, the proof of Lemma~\ref{lem:n-hom-finite} can be repeated verbatim.  For $\cN$, we may assume that $\cF = \oDelta_\lambda[m]$ and $\cG = \onabla_\mu[n]$.  If $\lambda \ne \mu$, the reasoning in Lemma~\ref{lem:n-hom-finite} applies again.  If $\lambda = \mu$, then by Lemma~\ref{lem:pre-exc-basic}, we may replace $\cF$ by $\onabla_\lambda[m]$.  By Lemma~\ref{lem:end-onabla}, the space
\[
\Hom(\onabla_\lambda[m], \onabla_\lambda[n]\{k\}) = \Hom(\onabla_\lambda, \onabla_\lambda[n-m+k]\la -k\ra)
\]
vanishes unless
\[
n-m+k = k = 0
\qquad\text{or}\qquad
\begin{cases}
n-m+k > 0, \\
-k \le -2n+2m -2k.
\end{cases}
\]
The latter condition is equivalent to $m - n < k \le 2m - 2n$.  In particular, only finitely many integers $k$ satisfy this condition, so our sum is finite-dimensional.
\end{proof}

\begin{prop}\label{prop:nilpotent-exc}
The set of all $\onabla_\lambda$ with $\lambda \in \bXp$ forms a dualizable graded pre-exceptional set in $\Db\Coh^{G\times \Gm}(\cN)$ with respect to both $\la 1\ra$ and $\{1 \}$, and its dual set consists of all
 $\oDelta_\lambda$ with $\lambda \in \bXp$.  Moreover, this collection is:
\begin{enumerate}
\item graded quasi-exceptional with respect to $\la 1\ra$, and
\item graded co-quasi-exceptional with respect to $\{1\}$.
\end{enumerate}
\end{prop}
\begin{proof}
As noted earlier, the fact that $\{\onabla_\lambda\}_{\lambda \in \bXp}$ is graded quasi-exceptional with respect to $\la 1\ra$ with dual set $\{\oDelta_\lambda\}_{\lambda \in \bXp}$ can be found in~\cite[Proposition~6.1]{achar-perv}.  Let us now show that this collection is graded co-quasi-exceptional with respect to $\{1\}$.  It is immediate (thanks to the $\la 1\ra$ case) that $\{\onabla_\lambda\}_{\lambda \in \bXp}$ satisfies conditions~\eqref{it:exc-nabla-hom} and~\eqref{it:exc-nabla-gen} of Definition~\ref{defn:excep}, and that $\End(\onabla_\lambda) \cong \bk$.  It remains to show that
\[
\Hom(\onabla_\lambda, \onabla_\lambda[n]\{k\}) = 0
\qquad\text{if $n = 0$ and $k \ne 0$, or if $n > 0$ and $k$ is arbitrary.}
\]
By Lemma~\ref{lem:end-onabla}, if the space $\Hom(\onabla_\lambda, \onabla_\lambda[n]\{k\}) = \Hom(\onabla_\lambda, \onabla_\lambda[n+k]\la -k\ra)$ is nonzero, we either have $n = k = 0$ or $n+k > 0$ and $-k \le -2n -2k$.  The latter condition can be rewritten as $n > -k$ and $2n \le -k$.  This implies that $2n < n$, and hence $n < 0$.
\end{proof}

In view of Proposition~\ref{prop:nilpotent-exc}, we may consider the following notions:
\begin{enumerate}
\item The \emph{exotic $t$-structure}, also called the \emph{perverse-coherent $t$-structure}, obtained by applying Proposition~\ref{prop:preexcep-struc}\eqref{it:preexcep-t} to $\{ \onabla_\lambda\}_{\lambda \in \bXp}$ with respect to $\la 1\ra$.
\item The \emph{supportive co-$t$-structure}, obtained by applying Proposition~\ref{prop:preexcep-struc}\eqref{it:preexcep-cot} to $\{ \onabla_\lambda\}_{\lambda \in \bXp}$ with respect to $\{ 1\}$.
\end{enumerate}
In fact, the objects $\fL_\lambda$ are (up to $\la 1\ra$) precisely the simple objects in the heart of the exotic $t$-structure: see~\cite[Proposition~2.6]{achar-vogan}.  The following statement is an immediate consequence of Lemma~\ref{lem:exotic-push}.

\begin{cor}\label{cor:cot-exact-push}
The functor $\pi_*: \Db\Coh^{G\times \Gm}(\tcN) \rightarrow \Db\Coh^{G\times \Gm}(\cN)$ is co-t-exact 
with respect to the supportive co-$t$-structures.
\end{cor}

We will now get to the first major application of the theory of co-$t$-structures in this setting. 
\begin{thm}\label{thm:full-push}
If $\cF \in \Db\Coh^{G\times \Gm}(\tcN)_{\geq 0}$ and $\cG \in \Db\Coh^{G\times \Gm}(\tcN)_{\leq 0}$, then 
the natural homomorphism
\[
\Hom(\cF, \cG) \rightarrow \Hom(\pi_*\cF, \pi_*\cG)
\]
is surjective. 
\end{thm}
\begin{proof}
We begin by defining two sets of objects $\tDelta, \, \tnabla \subset \Db\Coh^{G\times \Gm}(\tcN)$ by
\[
\begin{aligned}
\tDelta &= \{\tDelta_\lambda\{k\}[i] \mid \lambda \in \bX, \, i \leq 0, \, k \in \Z\},\\
\tnabla &= \{ \tnabla_\lambda\{k\}[i] \mid \lambda \in \bX, \, i \geq 0, \, k \in \Z\}.
\end{aligned}
\]
We will now use the $\uast$ operation from \S\ref{ss:tate-twist} with respect to $\{1\}$.  The description of the co-$t$-structure in Proposition~\ref{prop:preexcep-struc}\eqref{it:preexcep-cot} shows that
\begin{align*}
\Db\Coh^{G\times\Gm}(\tcN)_{\geq 0} &= \text{direct summands of $\bigcup_{n \ge 1} \tDelta^{\uast n}$,} \\
\Db\Coh^{G\times\Gm}(\tcN)_{\leq 0} &= \text{direct summands of $\bigcup_{n \ge 1} \tnabla^{\uast n}$.}
\end{align*}
Thus, it suffices to verify the surjectivity of the morphism 
\begin{equation}\label{eqn:natural-morphism}
\Hom(\cF, \cG) \rightarrow \Hom(\pi_*\cF, \pi_*\cG)
\end{equation}
with $\cF \in \tDelta^{\uast n}$ and $\cG \in \tnabla^{\uast m}$, for some $n,m \ge 1$. We will proceed by double induction with respect to $n$ and $m$. 

First suppose that $n=m=1$ and assume
$\cF = \tDelta_\lambda\{k_1\}[i_1]$ and $\cG = \tnabla_\mu\{k_2\}[i_2]$ with $\lambda,\mu \in \bX$, $k_1, k_2 \in \Z$, $i_1 \leq 0$ and $i_2 \geq 0$. Set $k = k_2 - k_1$ and $i = i_2-i_1 \geq 0$. We have 
\[
\Hom(\cF, \cG) = \Hom(\tDelta_\lambda,  \tnabla_\mu\{k\}[i]) = \begin{cases} 
	\bk & \text{if $\lambda = \mu$, $i = k = 0$},\\
	0 & \text{otherwise.}
	\end{cases}
\]
Observe that if ${\dom \lambda} \neq {\dom \mu}$, then both sides of \eqref{eqn:natural-morphism} are zero, so the map is trivially surjective. 
Now suppose that ${\dom \lambda} = {\dom \mu}$, and let $\delta = \delta^*_\mu + \delta^*_\lambda$.  Using Lemma~\ref{lem:exotic-push} and the fact that $\{\onabla_\lambda\}_{\lambda \in \bXp}$ is co-quasi-exceptional, we have 
\begin{equation*}\label{eqn:del-nabla-map}
\begin{aligned}
\Hom(\pi_*\cF, \pi_*\cG) &= 
   \Hom(\oDelta_{\dom \lambda}, \onabla_{\dom \lambda}\la \delta\ra\{k\}[i])\\
  &= \Hom(\oDelta_{\dom \lambda}, \onabla_{\dom \lambda}\{k-\delta\}[i+\delta])\\
  &= \begin{cases}
  	\bk & \text{if $k=\delta$, $i = -\delta$},\\
	0 & \text{otherwise (since $i + \delta \ge 0$)}.
	\end{cases}
  \end{aligned}
\end{equation*}
So it suffices to assume that $k= \delta = -i$ because the natural map will again be trivially surjective otherwise. 
But we have that both $i\geq 0$ and $\delta = \delta^*_\lambda + \delta^*_\mu\geq 0$, and thus,
$\delta = i = 0$. In particular, 
$\delta^*_\lambda = \delta^*_\mu = 0$ which further implies that $\lambda = \mu = w_0({\dom \lambda}) \in -\bX^+$. Hence,
\eqref{eqn:natural-morphism} is given by
\[
\Hom(\tDelta_\lambda,\tnabla_\lambda) \rightarrow \Hom(\oDelta_{\dom \lambda},\onabla_{\dom \lambda}).
\]
Any nonzero map $\tDelta_\lambda \to \tnabla_\lambda$ factors through $\tfL_\lambda = \tfL_{w_0(\dom \lambda)}$, and any nonzero map $\oDelta_{\dom \lambda} \to \onabla_{\dom \lambda}$ factors through $\fL_{\dom \lambda}$.  Lemma~\ref{lem:exotic-push} then implies that this map of $\Hom$-groups is nonzero, and hence an isomorphism
since both sides are isomorphic to $\bk$. We have established surjectivity of~\eqref{eqn:natural-morphism} in the case where $n = m = 1$.


Now suppose \eqref{eqn:natural-morphism} is surjective for $n =1$ and some $m\geq 1$. Let  
$\cF \in \tDelta$ and $\cG \in \tnabla^{\uast m+1}$ be arbitrary. Then
 $\cG$ fits into a distinguished triangle of the form
\[
\cG' \rightarrow \cG \rightarrow \cG'' \rightarrow,
\]
where $\cG' \in \tnabla^{\uast m}$ and $\cG'' \in \tnabla$. Applying $\Hom(\cF,-)$ and $\Hom(\pi_*\cF,-)$ we obtain the commutative diagram
\[
\hbox{\small$\begin{tikzcd}[column sep=small]
 \Hom(\cF, \cG') \ar[d, "f_1"] \ar[r] & \Hom(\cF, \cG) \ar[r] \ar[d, "f_2"] & \Hom(\cF, \cG'') \ar[r]\ar[d, "f_3"] & \Hom(\cF, \cG'[1]) \ar[d, "f_4"] \\
 \Hom(\pi_*\cF, \pi_*\cG') \ar[r] & \Hom(\pi_*\cF, \pi_*\cG) \ar[r] & \Hom(\pi_*\cF, \pi_*\cG'') \ar[r] & \Hom(\pi_*\cF, \pi_*\cG'[1]).
\end{tikzcd}$}
\]
Our inductive hypothesis implies that $f_1$ and $f_3$ are surjective, while Corollary~\ref{cor:cot-exact-push} and 
Definition~\ref{defn:cotstruc}(\ref{ax:hom}) imply 
\[
0 = \Hom(\cF, \cG'[1])  = \Hom(\pi_*\cF, \pi_*\cG'[1]).
\]
It then follows from the four lemma that $f_2$ is surjective, and so we are done with this case. 

Finally, fix $n \geq 1$ and assume that \eqref{eqn:natural-morphism} is surjective for any $m \geq 1$ with 
$\cF \in \tDelta^{\uast n}$ and $\cG \in \tnabla^{\uast m}$. Let $\cF \in \tDelta^{\uast n+1}$ and $\cG \in \tnabla^{\uast m}$ be arbitrary. 
Again, there is a distinguished triangle
\[
\cF' \rightarrow \cF \rightarrow \cF'' \rightarrow
\]
where $\cF' \in \tDelta^{\uast n}$ and $\cF'' \in \tDelta$. By applying $\Hom(-,\cG)$, $\Hom(-,\pi_*\cG)$ and then proceeding exactly as above, 
we can again
deduce the surjectivity of  \eqref{eqn:natural-morphism}. We are finished by induction. 
\end{proof}

\begin{thm}\label{thm:indec-silt}
Let $\cS^{G \times \Gm}(\cN)$ denote the silting subcategory of $\Db\Coh^{G\times\Gm}(\cN)$.
\begin{enumerate}
\item For any $\mu \in \bX$, the object $\pi_*\tfE_\mu$ is either $0$ or an indecomposable object of $\cS^{G \times \Gm}(\cN)$.  It is nonzero if and only if $\mu \in -\bX^+$.\label{it:indec-vanish}
\item For $\lambda \in \bX^+$, set $\fE_\lambda := \pi_*\tfE_{w_0\lambda}$.  Every indecomposable object of $\cS^{G \times \Gm}(\cN)$ is isomorphic to an object of the form 
$\fE_\lambda\{r\}$ for some $\lambda \in \bX^+$ and $r \in \Z$.\label{it:indec-class}
\end{enumerate} 
\end{thm}
The vanishing of $\pi_*\tfE_\mu$ for $\mu \notin -\bX^+$ was previously known: see, for instance, the comments in~\cite[\S4.2]{achar-hardesty-riche}.
\begin{proof}
For any $\mu \in \bX$, Corollary~\ref{cor:cot-exact-push} tells us that $\pi_*\tfE_\mu$ lies in $\cS^{G \times \Gm}(\cN)$, and by Theorem~\ref{thm:full-push}, $\pi_*$ induces a surjective ring homomorphism
\[
\End(\tfE_\mu) \twoheadrightarrow \End(\pi_*\tfE_\mu).
\]
Here, $\End(\tfE_\mu)$ is a local ring, since $\tfE_\mu$ is an indecomopsable object in a Krull--Schmidt category.  Its quotient ring $\End(\pi_*\tfE_\mu)$ is therefore either $0$ or a local ring.  Since $\cS^{G \times \Gm}(\cN)$ is also Krull--Schmidt, we conclude that $\pi_*\tfE_\mu$ is either $0$ or indecomposable.

Suppose now that $\mu \in -\bX^+$, and let $\lambda = w_0\mu \in \bX^+$. According to Proposition~\ref{prop:preexcep-struc}, the two compositions 
\[
\tDelta_{w_0\lambda} \to \tfL_{w_0\lambda} \to \tnabla_{w_0\lambda}
\qquad\text{and}\qquad
\tDelta_{w_0\lambda} \to \tfE_{w_0\lambda} \to \tnabla_{w_0\lambda}
\]
are equal (and nonzero).  Now apply $\pi_*$ to these maps.  Using Lemma~\ref{lem:exotic-push}, we obtain
\begin{equation}\label{eqn:fe-classify}
\oDelta_\lambda \to \fL_\lambda \to \onabla_\lambda
\qquad\text{and}\qquad
\oDelta_\lambda \to \fE_\lambda \to \onabla_\lambda.
\end{equation}
These compositions are again equal, and the first one is nonzero, so the second one is as well.  In particular, $\fE_\lambda \ne 0$.  


Part~\eqref{it:indec-class} of the theorem now follows from the ``abstract'' classification of indecomposable silting objects given in Proposition~\ref{prop:preexcep-struc}, together with the equality of the two compositions in~\eqref{eqn:fe-classify}.

To finish the proof of part~\eqref{it:indec-vanish}, we must show that if $\pi_*\tfE_\mu \ne 0$, then $\mu \in -\bX^+$.  By part~\eqref{it:indec-class}, if $\pi_*\tfE_\mu$ is nonzero, it is isomorphic to $\fE_\lambda\{r\}$ for some $\lambda \in \bX^+$ and some $r \in \Z$, so there are nonzero maps
\[
\oDelta_\lambda\{r\} \to \pi_*\tfE_\mu \to \onabla_\lambda\{r\}
\]
whose composition is also nonzero.  By Theorem~\ref{thm:full-push}, these maps arise by applying $\pi_*$ to some nonzero maps
\[
\tDelta_{w_0\lambda}\{r\} \to \tfE_\mu \to \tnabla_{w_0\lambda}\{r\}
\]
whose composition is again nonzero.  By Proposition~\ref{prop:preexcep-struc}, we conclude that $\mu = w_0\lambda$ and $r = 0$, as desired.
\end{proof}

\section{Push-forward functors and dg-coherent sheaves}
\label{sec:pushforward}

In this section, we prove some additional lemmas on complexes of coherent sheaves on $\tcN_I$ and $\cN$.  We also introduce ``dg versions'' of the main geometric categories.

\subsection{More functors}
\label{ss:more-functors}

Recall that $\tcN_I = G \times^{P_I} \fn_I$.  We also set $\tcN^I := G \times^B \fn_I$, and we define maps
\[
\pi_I : \tcN_I \to \cN, \qquad
\mu_I : \tcN^I \to \tcN_I, \qquad
e_I : \tcN^I \hookrightarrow \tcN
\]
as follows: $\pi_I(g,x) = \mathrm{Ad}(g)(x)$, $\mu_I$ is the quotient map, and $e_I$ is induced by the inclusion map $\fn_I \hookrightarrow \fn$.  Note that $\pi_\varnothing$ is the Springer resolution $\pi: \tcN \to \cN$, while $\mu_\varnothing$ and $e_\varnothing$ are identity maps.  The square
\begin{equation}\label{eqn:allmaps}
\begin{tikzcd}
\tcN^I \ar[r, "e_I"] \ar[d, "\mu_I"'] & \tcN \ar[d, "\pi"] \\
\tcN_I \ar[r, "\pi_I"] & \cN
\end{tikzcd}
\end{equation}
commutes (but it is not cartesian).  All four maps are proper, so the derived functors $\pi_{I*}$, $\mu_{I*}$, etc., take bounded complexes of (equivariant) coherent sheaves to bounded complexes of (equivariant) coherent sheaves.  The map $\mu_I$ is smooth, and the map $e_I$ is the inclusion of one smooth variety in another, so the same comments apply to $\mu_I^!$ and $e_I^*$ (see also \cite[\S 9.2]{achar-riche}). We define functors
\begin{align*}
\Xi_I := \mu_{I*} e_I^*&: \Db\Coh^{G \times \Gm}(\tcN) \to \Db\Coh^{G \times \Gm}(\tcN_I), \\
\Xi^I := e_{I*}\mu_I^!&: \Db\Coh^{G \times \Gm}(\tcN_I) \to \Db\Coh^{G \times \Gm}(\tcN).
\end{align*}

Finally, let
\[
s_I: \fn_I \hookrightarrow \cN
\]
be the inclusion map.  Since the $P_I$-action on $\cN$ extends to an action of $G$, there is a well-defined functor
\[
R\Ind_{P_I}^G: \Db\Coh^{P_I \times \Gm}(\cN) \to \Db\Coh^{G \times \Gm}(\cN).
\]
By construction, the following diagram commutes:
\[
\begin{tikzcd}[column sep=0pt]
\Db\Coh^{G \times \Gm}(\tcN_I) \ar[dr, "\pi_{I*}"'] \ar[rr, "\sim"] &&
\Db\Coh^{P_I \times \Gm}(\fn_I) \ar[dl, "R\Ind_{P_I}^G \circ s_{I*}"] \\
& \Db\Coh^{G \times \Gm}(\cN)
\end{tikzcd}
\]

\begin{lem}\label{lem:xi-commute}
The following diagram commutes:
\[
\begin{tikzcd}[column sep=0pt]
\Db\Coh^{G \times \Gm}(\tcN_I) \ar[dr, "\pi_{I*}"'] \ar[rr, "\Xi^I"] &&
\Db\Coh^{G \times \Gm}(\tcN) \ar[dl, "\pi_*"] \\
& \Db\Coh^{G \times \Gm}(\cN)
\end{tikzcd}
\]
\end{lem}
\begin{proof}
Let
\[
\rho_I := \frac{1}{2}\sum_{\alpha \in \Phi_I^+}\alpha \qquad\text{and}\qquad 
r_I := \dim P_I/B.
\]
Let $\cO_{\tcN^I}(-2\rho_I)$ denote the pullback along $\tcN^I \to G/B$ of the line bundle of weight $-2\rho_I$.  By smooth base change for the square
\[
\begin{tikzcd}
\tcN^I \ar[r] \ar[d, "\mu_I"'] & G/B = G \times^{P_I} (P_I/B) \ar[d]  \\
\tcN_I \ar[r] & G/P_I
\end{tikzcd}
\]
we see that $\mu_{I*}\cO_{\tcN^I}(-2\rho_I)$ is the pullback of the object in $\Db\Coh^{G \times \Gm}(G/P_I)$ corresponding to $R\Ind_B^{P_I}(-2\rho_I)$.  Using the form of Serre duality explained in~\cite[\S II.4.2(8)]{jantzen}, we have $R\Ind_B^{P_I}(-2\rho_I) \cong R\Ind_B^{P_I}(\bk)[-r_I] \cong \bk[-r_I]$, so
\[
\mu_{I*}\cO_{\tcN^I}(-2\rho_I)[r_I] \cong \cO_{\tcN_I}.
\]
Finally, we recall (cf. \cite[Lemma~9.4]{achar-riche}) that
\[
\mu_I^!(-) \cong \mu_I^*(-)\lotimes \cO_{\tcN^I}(-2\rho_I)[r_I].
\]

Using these observations, the commutativity of~\eqref{eqn:allmaps}, and the projection formula, for $M \in \Db\Coh^{G \times \Gm}(\tcN_I)$, we have
\begin{multline*}
\pi_*\Xi^I(M) = \pi_* e_{I*} \mu_I^!(M)
\cong \pi_{I*}\mu_{I*}(\mu_I^*(M) \lotimes_{\tcN^I}\cO_{\tcN^I}(-2\rho_I)[r_I]) \\
\cong \pi_{I*}(M \lotimes \mu_{I*}\cO_{\tcN^I}(-2\rho_I)[r_I]) \cong \pi_{I*}(M),
\end{multline*}
as desired.
\end{proof}

\subsection{dg coherent sheaves}

One goal of this section and the next is to show that the functor $\pi_{I*}: \Db\Coh^{G \times \Gm}(\tcN_I) \to \Db\Coh^{G \times \Gm}(\cN)$ is co-$t$-exact (see Corollary~\ref{cor:para-co-t}).  However, the proof involves a construction that takes place not in the setting of ordinary coherent sheaves, but in the setting of \emph{dg coherent sheaves} (see the comments in~\S\ref{ss:steinberg}).  In this subsection, we introduce notation related to dg coherent sheaves on some of the varieties we have discussed above.

See Appendix~\ref{app:dg} for generalities on (equivariant) dg coherent sheaves on an affine scheme.  In the framework of that appendix, one may consider the categories
\[
D^{P_I}_\Coh(\fn_I),\qquad
D^B_\Coh(\fn_I),\qquad
D^G_\Coh(\cN).
\]
It will be convenient to also be able to speak of ``dg coherent sheaves on $\tcN_I$ or $\tcN^I$.'' Since these varieties are not affine, they are not covered by the theory of Appendix~\ref{app:dg}.  Instead, we just take analogues of~\eqref{eqn:restric-equiv} as definitions: we set
\[
D^G_\Coh(\tcN_I) := D^{P_I}_\Coh(\fn_I)
\qquad\text{and}\qquad
D^G_\Coh(\tcN^I) := D^B_\Coh(\fn_I)
\]
We have ``degrading functors'' (see Proposition~\ref{prop:degrade})
\begin{align*}
\xi_I &: \Db\Coh^{G \times \Gm}(\tcN_I) \to D^G_\Coh(\tcN_I), \\
\xi_\cN &: \Db\Coh^{G \times \Gm}(\cN) \to D^G_\Coh(\cN).
\end{align*}
(Such a functor exists for $\tcN^I$ as well, of course, but we will not need a separate notation for it.)

Using Lemmas~\ref{lem:dg-pushpull} and~\ref{lem:dg-indres}, one can define the functors
\begin{align*}
\bar\pi_{I*} = R\Ind_{P_I}^G \bar s_{I*}&: D^G_\Coh(\tcN_I) \to D^G_\Coh(\cN), \\
\bar \mu_{I*} = R\Ind_B^{P_I}&: D^G_\Coh(\tcN^I) \to D^G_\Coh(\tcN_I), \\
\bar \mu_I^! = \Res_B^{P_I}({-})[n_I] \lotimes \textstyle\bigwedge^{n_I} (\fp_I/\fb)^*&: D^G_\Coh(\tcN_I) \to D^G_\Coh(\tcN^I), \\
\bar e_{I*}&: D^G_\Coh(\tcN^I) \to D^G_\Coh(\tcN), \\
\bar e_I^*&: D^G_\Coh(\tcN) \to D^G_\Coh(\tcN^I),
\end{align*}
where $n_I = \dim P_I - \dim B$.  Each of these functors is compatible with its ordinary (non-dg) analogue via the appropriate degrading functors.  That is, the diagram
\[
\begin{tikzcd}
\Db\Coh^{G \times \Gm}(\tcN_I) \ar[r, "\xi_I"] \ar[d, "\pi_{I*}"'] &
  D^G_\Coh(\tcN_I) \ar[d, "\ol{\pi}_{I*}"] \\
\Db\Coh^{G \times \Gm}(\cN) \ar[r, "\xi_\cN"] &
  D^G_\Coh(\cN)
\end{tikzcd}
\]
commutes, and likewise for each of the other functors defined above.

Lastly, we define
\begin{align*}
\bar\Xi_I := \bar\mu_{I*} \bar e_I^*&: D^G_\Coh(\tcN) \to D^G_\Coh(\tcN_I), \\
\bar\Xi^I := \bar e_{I*}\bar\mu_I^!&: D^G_\Coh(\tcN_I) \to D^G_\Coh(\tcN).
\end{align*}
By the techniques of Appendix~\ref{app:dg}, the proof of Lemma~\ref{lem:xi-commute} can be adapted to show that the diagram
\begin{equation}\label{eqn:xi-commute-dg}
\begin{tikzcd}[column sep=0pt]
D^G_\Coh(\tcN_I) \ar[dr, "\bar\pi_{I*}"'] \ar[rr, "\bar\Xi^I"] &&
D^G_\Coh(\tcN) \ar[dl, "\bar\pi_*"] \\
& D^G_\Coh(\cN)
\end{tikzcd}
\end{equation}
commutes.

\begin{lem}\label{lem:co-t-structure-N}
Let $I \subset S$ be arbitrary.
\begin{enumerate}
\item The set $\{\xi_I(\tfE_{I,\lambda})\}_{\lambda \in \bXpp_I}$ generates a silting subcategory of 
$D^G_\Coh(\tcN_I)$, so that $\xi_I$ becomes co-$t$-exact.  Moreover, an object $\cF \in \Db\Coh^{G \times \Gm}(\tcN_I)$ is silting if and only if $\xi_I(\cF) \in D^G_\Coh(\tcN_I)$ is silting.\label{it:xiI-cot}
\item The set $\{\xi_\cN(\fE_{\lambda})\}_{\lambda \in \bX^+}$ generates a silting subcategory of 
$D^G_\Coh(\cN)$, so that both $\xi_{\cN}$ and 
$\ol{\pi}_*: D^G_\Coh(\tcN) \rightarrow D^G_\Coh(\cN)$ become co-$t$-exact functors. Moreover, an object $\cF \in \Db\Coh^{G \times \Gm}(\cN)$ is silting if and only if $\xi_{\cN}(\cF) \in D^G_\Coh(\cN)$ is silting.\label{it:xiN-cot}
\end{enumerate}
\end{lem}
\begin{proof}
\eqref{it:xiI-cot}~Let $\cS^G(\tcN_I)$ denote the full additive subcategory of $D^G_\Coh(\tcN_I)$ consisting of (finite) direct sums of objects of the form $\xi_I(\tfE_{I,\lambda})$.  We first claim that each $\xi_I(\tfE_{I,\lambda})$ is indecomposable.  By Proposition~\ref{prop:degrade}\eqref{it:graded-ff}, we have
\[
\End(\xi_I(\tfE_{I,\lambda})) \cong \bigoplus_{n \in \Z} \Hom(\tfE_{I,\lambda}, \tfE_{I,\lambda}\{n\}).
\]
These spaces are finite-dimensional by Lemma~\ref{lem:n-hom-finite-2}.  The right-hand side can be thought of as a graded artinian ring.  Its degree-$0$ component $\End(\tfE_{I,\lambda})$ is a local ring, so by~\cite[Theorem~3.1]{gordon-green}, $\End(\xi_I(\tfE_{I,\lambda}))$ is also local.  We conclude that $\xi_I(\tfE_{I,\lambda})$ is indecomposable, and hence that $\cS^G(\tcN_I)$ is closed under direct summands.

Next, we claim that part~\eqref{it:silting-ext} of Definition~\ref{defn:silting-cat} holds for $\cS^G(\tcN_I)$.  This follows from the corresponding property for $\cS^{G \times \Gm}(\tcN_I)$, together with Proposition~\ref{prop:degrade}\eqref{it:graded-ff} again.

Since $\Db\Coh^{G \times \Gm}(\tcN_I) = \lgen \cS^{G \times \Gm}(\tcN_I) \rgen_\tri$ (see Remark~\ref{rmk:silting-gen}), using Proposition~\ref{prop:degrade}\eqref{it:graded-surj}, we see that $D^G_\Coh(\tcN_I) = \lgen \cS^G(\tcN_I) \rgen_\tri$.  Thus $\cS^G(\tcN_I)$ is a silting subcategory of $D^G_\Coh(\tcN_I)$, and $\xi_I$ is co-$t$-exact.

Finally, it remains to prove that if $\xi_I(\cF)$ is silting, then $\cF$ is silting.  Suppose $\cF$ is \emph{not} silting.  Using the description of the co-$t$-structure given in Proposition~\ref{prop:silt_co-t-structure}, we see there must either be a nonzero morphism $\cF \to \tfE_{I,\lambda}[n]\la k\ra$ for some $\lambda$ and some $n > 0$, or a morphism $\tfE_{I,\lambda}[-n]\la k\ra \to \cF$, again with $n > 0$.  Since $\xi_I$ is faithful (Proposition~\ref{prop:degrade}\eqref{it:graded-ff}), applying $\xi_I$ to this morphism yields a nonzero morphism in $D^G_\Coh(\tcN_I)$ that shows that $\xi_I(\cF)$ is not silting either.

\eqref{it:xiN-cot}~Most of this statement holds by the same reasoning as in part~\eqref{it:xiI-cot}; the only issue to address is the claim that $\ol{\pi}_*$ is co-$t$-exact.  This follows from the corresponding claim for $\pi_*$ (Corollary~\ref{cor:cot-exact-push}).
\end{proof}

\section{Representations of \texorpdfstring{$G$}{G}}
\label{sec:repG}

From now on, we assume that the characteristic $p$ of $\bk$ is larger than the Coxeter number $h$ for $G$, and that $G$ has a simply connected derived subgroup.  Let $\mbf{G}$ be such that $G = \mbf{G}/\mbf{G}_{1}$, where $\mbf{G}_{1}$ is the Frobenius kernel 
(i.e. so that $G$ is the Frobenius twist of $\mbf{G}$). Let $\mbf{B} \subset \mbf{G}$ be the Borel subgroup corresponding to $B \subset G$.  Let 
$\Rep(\mbf{G})$ denote the category of finite-dimensional rational representations of $\mbf{G}$.  For $\lambda \in \bX^+$, let $\weyl(\lambda)$, resp.~$\coweyl(\lambda)$, resp.~$\tilt(\lambda)$, denote the Weyl module, resp.~dual Weyl module, resp.~indecomposable tilting module, of highest weight $\lambda$.

For each $I \subset S$, fix a choice of weight $\varsigma_I \in \bX^+$ whose pairing with simple coroots is given by
\[
\alpha_s^\vee(\varsigma_I) =
\begin{cases}
1 & \text{if $s \in I$,} \\
0 & \text{if $s \notin I$.}
\end{cases}
\]
(The existence of such a weight is guaranteed by the assumption that the derived subgroup of $G$ is simply connected.)  The
\emph{extended block} $\Rep_I(\mbf{G}) \subset \Rep(\mbf{G})$ is defined as the Serre subcategory generated by all of
the simple modules whose highest weight is contained in $\bX^+ \cap \Waff\cdot_p (-\varsigma_I)$, where ``$\cdot_p$'' denotes the $p$-dilated dot action (see, for instance, \cite[\S 1.2]{achar-hardesty-riche} for a discussion).  This subcategory is  a direct summand of $\Rep(\mbf{G})$, and there is a projection functor
\[
\pr_I: \Rep(\mbf{G}) \to \Rep_I(\mbf{G}).
\]
One of the main results of~\cite{achar-riche} can be restated as follows.

\begin{thm}[\cite{achar-riche}]\label{thm:ar}
There is an equivalence of categories
\[
F_I: D^G_\Coh(\tcN_I) \xrightarrow{\sim} \Db\Rep_I(\mbf{G}). 
\]
Moreover, this functor satisfies
\[
F_I(\xi_I(\tnabla_{I,\lambda})) \cong \coweyl(w_\lambda \cdot_p (-\varsigma_I))
\qquad\text{and}\qquad
F_I(\xi_I(\tDelta_{I,\lambda})) \cong \weyl(w_\lambda \cdot_p (-\varsigma_I)),
\]
\end{thm}

This statement does not appear in quite this form in~\cite{achar-riche}, but it is easily deduced from~\cite[Propositions~10.3 and~10.6]{achar-riche}, which state that there is a degrading functor $\tilde F_I: \Db\Coh^{G \times \Gm}(\tcN_I) \to \Db\Rep_I(\mbf{G})$ with the desired behavior on $\tnabla_{I,\lambda}$ and $\tDelta_{I,\lambda}$.  More precisely, examining the construction of $\tilde F_I$, one sees that a crucial role is played by a certain ``bigraded Koszul duality functor'' $\kappa_I$, described in~\cite[\S 4]{achar-riche}.  After replacing this by a singly graded Koszul duality functor following~\cite[\S 8]{gkm}, one can check that there is an equivalence of categories $F_I$ that makes the following diagram commute:
\[
\begin{tikzcd}[column sep=0pt]
 \Db\Coh^{G \times \Gm}(\tcN_I) \ar[dr, "\tilde F_I"'] \ar[rr, "\xi_I"] && D^G_\Coh(\tcN_I) \ar[dl, "F_I"] \\
 & \Db\Rep_I(\mbf{G})
 \end{tikzcd}
 \]

\begin{rmk}\label{rmk:F-cot-exact}
If we equip $D^G_\Coh(\tcN_I)$ with the co-$t$-structure from Lemma~\ref{lem:co-t-structure-N} and $\Db\Rep_I(\mbf{G})$ with the co-$t$-structure whose coheart consists of tilting $\mbf{G}$-modules in $\Rep_I(\mbf{G})$, then the functor $F_I$ is co-$t$-exact.  To see this, use Lemma~\ref{lem:tcn-silt-tilt} and observe that the functor $\tilde F_I$ sends tilting objects in the heart of the representation-theoretic $t$-structure to tilting $\mbf{G}$-modules (cf.~\cite[Remark~3.7]{acr}).
\end{rmk}

One can transfer various representation-theoretic or geometric constructions across the equivalence of Theorem~\ref{thm:ar}.  In this section and the next one, we exploit this idea to obtain results on $\mbf{G}_{1}$-cohomology.

\subsection{Steinberg translation functors}
\label{ss:steinberg}

For each subset $I \subset S$, we define
\[
\St_I^{\mbf{G}} := \Ind_{\mbf{B}}^{\mbf{G}}((p-1)\varsigma_I) \in \Rep_I(\mbf{G}).
\]
It is easy to verify that $(p-1)\varsigma_I \in \bX^+ \cap \Waff\cdot_p (-\varsigma_I)$ is a minimal element, 
so by the linkage principle (see \cite[II.7]{jantzen}), we have 
\[
\St_I^{\mbf{G}} \cong \tilt((p-1)\varsigma_I) \in \Rep_I(\mbf{G}).
\]
(In particular, $\St_{\varnothing}^{\mbf{G}} \cong \bk$ if we set $\varsigma_\varnothing = 0$.)  We define the \emph{Steinberg translation functors} to be the functors given by
\[
\begin{aligned}
\bS^I := \pr_{\varnothing}\circ - \otimes (\St_I^{\mbf{G}})^*: \Db\Rep_I(\mbf{G}) &\longrightarrow \Db\Rep_{\varnothing}(\mbf{G}) \\
\bS_I := \pr_{I}\circ - \otimes \St_I^{\mbf{G}}: \Db\Rep_I(\mbf{G}) &\longleftarrow \Db\Rep_{\varnothing}(\mbf{G}).
\end{aligned}
\]
Observe that $\bS^I$ and $\bS_I$ are both left and right adjoints to each other. We can consider the transports of these functors across the equivalence of Theorem~\ref{thm:ar}:
\[
\begin{aligned}
\Sigma^I&:= {F_\varnothing}^{-1}\circ \pmb{\Sigma}^I\circ {F_I}: D^G_\Coh(\tcN_I) \rightarrow D^G_\Coh(\tcN), \\
\Sigma_I&:= {F_I}^{-1}\circ \pmb{\Sigma}_I\circ {F_\varnothing}: D^G_\Coh(\tcN) \rightarrow D^G_\Coh(\tcN_I).
\end{aligned}
\]
Unfortunately, we do not know how to give concrete geometric descriptions of $\Sigma^I$ or $\Sigma_I$ in the language of dg coherent sheaves, nor how to ``lift'' them to $\Db\Coh^{G \times \Gm}(\tcN_I)$.  (If such lifts were available, the main results of this paper could be proved using only the language of coherent sheaves, and avoiding the technical difficulties of dg coherent sheaves.)

On the other hand, we can also consider the transports 
\[
\begin{aligned}
\pmb{\Xi}^I &:= {F_\varnothing} \circ \bar\Xi^I \circ {F_I}^{-1}: \Db\Rep_I(\mbf{G})\rightarrow \Db\Rep_\varnothing(\mbf{G}),\\
\pmb{\Xi}_I &:= {F_I} \circ \bar\Xi_I \circ {F_\varnothing}^{-1} :\Db\Rep_I(\mbf{G}) \leftarrow \Db\Rep_\varnothing(\mbf{G}).
\end{aligned}
\]
We do not know how to give an explicit representation-theoretic interpretation of $\pmb{\Xi}^I$, but we will see some representation-theoretic information about $\pmb{\Xi}_I$ in the proof of Lemma~\ref{lem:hom-nat-iso} below.

\subsection{A dg enhancement of \texorpdfstring{$\mbf{G}_{1}$}{G1}-cohomology}

For each $I \subset S$, define a functor 
\[
\mbf{H}_I: \Db\Rep(\mbf{G}) \rightarrow D^G_\Coh(\cN)
\qquad\text{by}\qquad
\mbf{H}_I := \ol{\pi}_{I*}\circ F_I^{-1} \circ \pr_I.
\]
The following lemma explains the relationship between these functors and $\mbf{G}_{1}$-cohomology.

\begin{lem}
\phantomsection
\label{lem:varphi-coh}
\begin{enumerate}
\item There is a $G$-equivariant isomorphism of graded rings $\bk[\cN] \cong \Ext^\bullet_{\mbf{G}_{1}}(\bk,\bk)$.\label{it:varphi-coh-aj}
\item For any $I \subset S$ and $M \in \Db\Rep_{I}(\mbf{G})$, there is a natural isomorphism of graded $G$-equivariant $\bk[\cN]$-modules\label{it:varphi-coh-gen}
\[
H^\bullet\mbf{H}_I(M) \cong \bigoplus_{k\in \Z} \Hom_{\mbf{G}_{1}}(\St_I^{\mbf{G}}, M[k]).
\]
\end{enumerate}
\end{lem}
\begin{proof}
Part~\eqref{it:varphi-coh-aj} is identical to~\cite[Lemma~8.1]{achar-hardesty-riche}. For part~\eqref{it:varphi-coh-gen}, copy the proofs of~\cite[Lemma~8.1 and Proposition~9.1]{achar-hardesty-riche}, using the observation that
 \[
{F_I}(\cO_{\tcN_I}) = \St_I^{\mbf{G}}. \qedhere
 \]
\end{proof}

\begin{lem}\label{lem:cohom-silting}
For any $\mu \in \bX^+$, we have
\[
\mbf{H}_\varnothing(\tilt(\mu)) \cong \begin{cases}
	\xi_\cN(\fE_{w_0\lambda}) & \text{if $\mu = w_\lambda \cdot_p 0$ for some $\lambda \in -\bX^+$,}\\
	0 	& \text{otherwise}. 
	\end{cases}
\]
\end{lem}
\begin{proof}
If $\mu$ is not of the form $w_\lambda \cdot_p 0$, then $\pr_\varnothing(\tilt(\mu)) = 0$, and the claim is obvious.  Assume now that $\mu = w_\lambda \cdot_p 0$ for some $\lambda \in \bX$.  As explained in the proof of~\cite[Proposition~9.1]{achar-hardesty-riche}, We have $F_\varnothing(\xi_\varnothing(\tfE_\lambda)) \cong \tilt(\mu)$, and hence
\[
\mbf{H}_\varnothing(\tilt(\mu)) \cong \ol{\pi}_* \xi_\varnothing(\tfE_\lambda) \cong \xi_\cN(\pi_* \tfE_\lambda) \cong
\begin{cases}
\xi_\cN( \fE_{w_0\lambda}) & \text{if $\lambda \in -\bX^+$,} \\
0 & \text{otherwise,}
\end{cases}
\]
where the last equality holds by Theorem~\ref{thm:indec-silt}.
\end{proof}

\begin{rmk}
Combining the preceding lemma with Theorem~\ref{thm:indec-silt} yields the following remarkable observation: the cohomology of an \emph{indecomposable tilting module} coincides with the cohomology of a uniquely
determined \emph{indecomposable silting object} of 
$D^G_\Coh(\cN)$.
\end{rmk}

\subsection{Main results}

The following lemma describes the relationship between the various functors introduced so far in this section.

\begin{lem}\label{lem:hom-nat-iso}
There exists a natural transformation 
\[
\pmb{\eta}: \pmb{\Xi}^I  \rightarrow \pmb{\Sigma}^I
\]
of functors $\Db\Rep_I(\mbf{G}) \to \Db\Rep_\varnothing(\mbf{G})$ such that the induced transformation
\[
\mbf{H}_\varnothing\pmb{\eta}:  
\mbf{H}_\varnothing\pmb{\Xi}^I \rightarrow \mbf{H}_\varnothing\pmb{\Sigma}^I,
\]
of functors $\Db\Rep_I(\mbf{G}) \to D^G_\Coh(\cN)$ is a natural isomorphism. 
\end{lem}
\begin{proof}
For this proof, we need slightly more information about $F_I$.  Let $\mbf{P}_I$ be the standard parabolic subgroup of $\mbf{G}$ corresponding to $I$. Then $F_I$ is the composition of two equivalences
\begin{equation}\label{eqn:induc-factorization}
D^G_\Coh(\tcN_I) \xrightarrow[\sim]{\psi_I} \Db_{\mathrm{Stein}}\Rep(\mbf{P}_I) \xrightarrow[\sim]{R\Ind_{\mbf{P}_I}^\mbf{G}} \Db\Rep_I(\mbf{G}).
\end{equation}
See~\cite[\S 5]{achar-riche} for the definition of $\Db_{\mathrm{Stein}}\Rep(\mbf{P}_I)$.  We will not need the details, except to note that this category is a full triangulated subcategory of $\Db\Rep(\mbf{P}_I)$.  According to~\cite[Propositions~7.5 and 9.25]{achar-riche} and their proofs (see especially~\cite[Figure~9]{achar-riche}), there is a commutative diagram
\[
\begin{tikzcd}
D^G_\Coh(\tcN) \ar[d, "\ol{\Xi}_I"'] \ar[r, "\psi_\varnothing"] & \Db_{\mathrm{Stein}}\Rep(\mbf{B}) \ar[d, "R\Ind_{\mbf{B}}^{\mbf{P}_I}(({-}) \otimes (p-1)\varsigma_I)"] \\
D^G_\Coh(\tcN_I) \ar[r, "\psi_I"] & \Db_{\mathrm{Stein}}\Rep(\mbf{P}_I)
\end{tikzcd}
\]
Using the fact that $F_\varnothing^{-1} \circ R\Ind_{\mbf{B}}^{\mbf{G}} \cong \psi_\varnothing^{-1}$, the commutativity of the diagram above can be expressed as the existence of a natural isomorphism
\[
\pmb{\Xi_I}\circ R\Ind_{\mbf{B}}^{\mbf{G}} \cong \pr_I\circ R\Ind_{\mbf{B}}^{\mbf{G}}( - \otimes (p-1)\varsigma_I): \Db_{\mathrm{Stein}}(\mbf{B}) \to \Db\Rep_I(\mbf{G}).
\]
On the other hand, it can be deduced from the tensor identity (see \cite[I.3.6]{jantzen}) that 
\[
\pmb{\Sigma}_I \circ R\Ind_{\mbf{B}}^{\mbf{G}} \cong \pr_I\circ R\Ind_{\mbf{B}}^{\mbf{G}}( - \otimes \St_I^{\mbf{G}}).
\]
(This isomorphism holds for all $\Db\Rep(\mbf{B})$, not just $\Db_{\mathrm{Stein}}\Rep(\mbf{B})$.)

Since $\St_I^{\mbf{G}} = \Ind_{\mbf{B}}^{\mbf{G}}((p-1)\varsigma_I)$,
 there exists a canonical non-zero map 
\[
p_I: \St_I^{\mbf{G}} \rightarrow (p-1)\varsigma_I,
\]
which corresponds to $\id_{\St_I^{\mbf{G}}}$ under Frobenius Reciprocity (see \cite[I.3.4]{jantzen}). 
From the formulas above, this map induces a natural transformation $\pmb{\Sigma}_I \circ R\Ind_{\mbf{B}}^{\mbf{G}} \to 
\pmb{\Xi_I}\circ R\Ind_{\mbf{B}}^{\mbf{G}}$ of functors $\Db_{\mathrm{Stein}}\Rep(\mbf{B}) \to \Db\Rep_I(\mbf{G})$.  Since $R\Ind_{\mbf{B}}^{\mbf{G}}: \Db_{\mathrm{Stein}}\Rep(\mbf{B}) \to \Db\Rep_\varnothing(\mbf{G})$ is an equivalence of categories (cf.~\eqref{eqn:induc-factorization}), we obtain a natural transformation
\[
\pmb{\gamma}:
	\pmb{\Sigma}_I \rightarrow \pmb{\Xi}_I. 
\]

Thus, we can define 
\[
\pmb{\eta}: \pmb{\Xi}^I  \longrightarrow \pmb{\Sigma}^I
\]
to be the natural transformation given by applying Lemma~\ref{lem:adjoint-trans} below to $\pmb{\gamma}$. 
To show that this induces
 the desired natural isomorphism, observe that 
 \[
{F_I}\,\Xi_I(\cO_{\tcN}) = {F_I}(\cO_{\tcN_I}) = \St_I^{\mbf{G}}.
 \]
 And since 
 $F_\varnothing^{-1}(\bk) = \cO_{\tcN}$, we get
 \[
 \mbf{\Xi}_I(\bk) = \St_I^{\mbf{G}}.
 \]
 Moreover, 
 \[
 \pmb{\gamma}_{\bk}: \pmb{\Sigma}_I(\bk) \xrightarrow{} \pmb{\Xi}_I(\bk)
 \]
 is an isomorphism which coincides with the identity map on $\St_I^{\mbf{G}}$. 
 
This leads to a commutative diagram of natural transformations
 \begin{equation*}
\begin{tikzcd}
\Hom_{\mbf{G}_{1}}(\bk,\pmb{\Xi}^I(-)) \ar[r, "\pmb{\eta}_*"] \ar[d, "\wr"']
    & \Hom_{\mbf{G}_{1}}(\bk,\pmb{\Sigma}^I(-)) \ar[d, "\wr"] \\
\Hom_{\mbf{G}_{1}}(\pmb{\Xi}_I(\bk),-) \ar[r, "\pmb{\gamma}^*", "\sim"'] &  \Hom_{\mbf{G}_{1}}(\pmb{\Sigma}_I(\bk),-),
\end{tikzcd}
\end{equation*}
where the vertical maps arise from the adjunction isomorphism.  
Notice that the transformation $\pmb{\eta}_*$ is a natural isomorphism since every other map in the diagram 
is an isomorphism. 

To see why  $\mbf{H}_\varnothing\pmb{\eta}$ is an isomorphism, notice that if we
 let  $M \in \Db\Rep_I(\mbf{G})$ be an arbitrary object, then it follows from Lemma~\ref{lem:varphi-coh}
  that for any $k \in \Z$,
 the morphism
\[
H^k\mbf{H}_\varnothing\pmb{\eta}: H^k\mbf{H}_\varnothing(\pmb{\Xi}^I(M)) \rightarrow 
 H^k\mbf{H}_\varnothing(\pmb{\Sigma}^I(M)),
\]
identifies with the isomorphism 
\[
\pmb{\eta}_*: \Hom_{\mbf{G}_{1}}(\bk, \pmb{\Xi}^I(M)[k]) \xrightarrow{\sim} \Hom_{\mbf{G}_{1}}(\bk, \pmb{\Sigma}^I(M)[k])
\]
given above. 
\end{proof}

\begin{lem}\label{lem:adjoint-trans}
Let $\fC$ and $\fD$ be categories with functors
\[
\xymatrix{
\fC \ar@<1ex>[r]^{F_1,\, F_2} 
	& \fD \ar@<1ex>[l]^{G_1,\, G_2} 
}
\]
such that $F_i$ is left adjoint to $G_i$ (for $i = 1,2$). Then the natural isomorphisms
\[
\Hom_\fD(F_i(-), -) \xrightarrow{\sim} \Hom_\fC(-,G_i(-)), \quad i=1,2
\]
induce a bijection
\[
\{\text{natural transformations $F_1 \to F_2$}\} \simto
\{\text{natural transformations $G_2 \to G_1$}\}.
\]
\end{lem}
\begin{proof}
This is a routine application of Yoneda's lemma. 
\end{proof}

The following statement, which is the main result of this section, is now essentially just a reformulation of the preceding lemma.

\begin{thm}\label{thm:pushforward-stein-commute}
For any $I\subset S$, there are natural isomorphisms 
\[
\ol{\pi}_*\circ \bar\Xi^I  \cong \ol{\pi}_*\circ \Sigma^I: D^G_\Coh(\tcN_I) \rightarrow D^G_\Coh(\cN).
\]
\end{thm}
\begin{proof}
By the definitions and Lemma~\ref{lem:hom-nat-iso}, we have
\[
\ol{\pi}_* \bar\Xi^I = \ol{\pi}_* F_\varnothing^{-1} \pmb{\Xi}^I F_I 
= \mbf{H}_\varnothing \pmb{\Xi}^I F_I \cong \mbf{H}_\varnothing \pmb{\Sigma}^I F_I \cong \ol{\pi}_* F_\varnothing^{-1}  \pmb{\Sigma}^I F_I = \ol{\pi}_*\Sigma^I.\qedhere
\]
\end{proof}


\begin{cor}\label{cor:para-co-t}
The functors $\pi_{I*}$ and $\ol{\pi}_{I*}$ are co-t-exact for any $I \subset S$. 
\end{cor}
\begin{proof}

We begin by showing that $\Sigma^I: D^G_\Coh(\tcN_I) \to D^G_\Coh(\tcN)$ is co-$t$-exact.  Since the equivalences $F_I$ and $F_\varnothing$ are co-$t$-exact (see Remark~\ref{rmk:F-cot-exact}), this claim is equivalent to showing that $\bS^I: \Db\Rep_I(\mbf{G}) \to \Db\Rep_\varnothing(\mbf{G})$ is co-$t$-exact.  Here, these derived categories are equipped with the co-$t$-structures whose cohearts consist of tilting $\mbf{G}$-modules.  In other words, we must show that $\bS^I$ sends tilting $\mbf{G}$-modules to tilting $\mbf{G}$-modules.  This claim is immediate from the definition and the fact that $\St_I^{\mbf{G}}$ is a tilting $\mbf{G}$-module.  Thus, $\Sigma^I$ is co-$t$-exact.

By~\eqref{eqn:xi-commute-dg}, we have $\ol{\pi}_{I*} \cong \ol{\pi}_* \bar \Xi^I$, so Theorem~\ref{thm:pushforward-stein-commute} gives $\ol{\pi}_{I*} \cong \ol{\pi}_*\circ \Sigma^I$.  Since $\ol{\pi}_*$ is co-$t$-exact (see Lemma~\ref{lem:co-t-structure-N}), we conclude that $\ol{\pi}_{I*}$ is co-$t$-exact.

It remains to prove the co-$t$-exactness of $\pi_{I*}$.  Lemma~\ref{lem:co-t-structure-N} implies that this is co-$t$-exact if and only if $\xi_\cN\pi_{I*}$ is co-$t$-exact.  The latter functor is isomorphic to $\ol{\pi}_{I*}\xi_I$, which we already know to be co-$t$-exact.
\end{proof}

\section{The scheme-theoretic Humphreys conjecture}
\label{sec:scheme-hum}

In this section, we continue to assume that $p>h$. Recall that for $\lambda \in \bX$, we write $w_\lambda$ for the element of minimal length in the coset $Wt_\lambda \subset \Waff$.  To state the Humphreys conjecture, we will need some facts about right Kazhdan--Lusztig cells in $\Waff$.\footnote{Kazhdan--Lusztig cells are usually considered in the context of the \emph{nonextended} affine Weyl group $W \ltimes \Z\Phi$, rather than in $\Waff$.  However, as explained in~\cite[Remark~2.1]{ctmap} it is straightforward to define Kazhdan--Lusztig cells in $\Waff$, and they are in bijection with those in $W \ltimes \Z\Phi$.} A right Kazhdan--Lusztig cell is called \emph{antispherical} if it contains some element of the form $w_\lambda$.  By results of Lusztig~\cite{lus:cawg4} and Lusztig--Xi~\cite{lx:clc}, the antispherical right cells are in bijection with the set of $G$-orbits in $\cN$.  Given a $G$-orbit $C \subset \cN$, let
\[
\bX_C = \{ \lambda \in \bX \mid \text{$w_\lambda$ lies in the antispherical right cell corresponding to $C$} \}.
\]
We thus obtain a partition of $\bX$ indexed by nilpotent orbits:
\[
\bX = \bigsqcup_{\substack{C \subset \cN \\ \text{a $G$-orbit}}} \bX_C.
\]
We also set $\bX^+_C := w_0(\bX_C) \cap \bX^+$. It follows from \cite[Theorem~6.4]{modular-LV} and~\cite[Remark~6]{bez-psaf}\footnote{In more detail,~\cite[Theorem~6.4]{modular-LV} says that the support of $\fL_\lambda$ is ``independent of $\bk$'' in an appropriate sense, and~\cite[Remark~6]{bez-psaf} says that the support of the analogue of $\fL_\lambda$ over $\C$ is the closure of the nilpotent orbit assigned to the antispherical cell containing $w_{w_0\lambda}$.  (Note that the definition of the notation ``$w_\lambda$'' in~\cite{bez-psaf} is different from ours; the $w_0$ does not appear in the statement there.)} 
that 
\[
\bX^+_C = \{ \lambda \in \bX^+ \mid \supp\fL_\lambda = \ol{C} \} \subset \bX^+,
\]
as long as $p$ is good. 
We will call $\bX^+_C$ the \emph{canonical cell} corresponding to $C$. 
If $Z \subseteq \cN$ is a $G$-stable subspace, we set 
\[
\bX_{Z} = \bigcup_{C \subset Z} \bX_{C}, \quad \bX^+_{Z} = \bigcup_{C \subset Z} \bX^+_{C}.
\]
\subsection{Scheme-theoretic Humphreys conjecture}
We begin by giving a scheme-theoretic analogue of the classical Humphreys conjecture.  Let
\[
\cC_p = \{ \lambda \in \bX \mid \text{$0 < \la \lambda+\rho,\alpha^\vee \ra < p$ for all positive roots $\alpha \in \Phi^+$} \}
\]
be the \emph{fundamental alcove}.  For any $w \in \Waff$, one can consider the set $w \cdot_p \cC_p$.  Any such set is called an \emph{alcove}.

For a nilpotent orbit $C \subset \cN$, let $\cI_{\ol{C}} \subset \bk[\cN]$ denote the defining ideal of the reduced subscheme $\ol{C} \subset \cN$.  We propose the following refinement of the conjecture proposed by Humphreys~\cite{hum:cmr} (see also~\cite[Conjecture~8.5]{achar-hardesty-riche} and the discussion preceding it):


\begin{conj}[Scheme-Theoretic Humphreys conjecture]\label{conj:classical-scheme-conj}
Suppose that $\mu \in \bX^+$ belongs to the lower closure of $w_{\lambda} \cdot_p \cC_p$ for some $\lambda \in \bX_C$. Then the annihilator of the $\bk[\cN]$-module $\Ext_{\mbf{G}_{1}}^*(\tilt(\mu), \tilt(\mu))$ is   $\mathcal{I}_{\ol{C}}$. 
\end{conj}
Following \cite[\S 8.3]{achar-hardesty-riche}, we also formulate a ``relative" version of this conjecture. 

\begin{conj}\label{conj:scheme-conj}
If $\lambda \in -\bX^+ \cap \bX_C$, then the annihilator of the $\bk[\cN]$-module $\Ext_{\mbf{G}_{1}}^*(\bk, \tilt(w_{\lambda}\cdot_p 0))$ is $\mathcal{I}_{\ol{C}}$.
\end{conj}

We emphasize that these conjectures are stronger than the conjectures on set-theoretic support that were proved in \cite{hardesty} and \cite{achar-hardesty-riche}.  Nevertheless, the set-theoretic ``lower bound'' proved in \cite[Theorem~9.3(1)]{achar-hardesty-riche} immediately implies a scheme-theoretic lower bound as well: namely, the annihilator in each of Conjectures~\ref{conj:classical-scheme-conj} and~\ref{conj:scheme-conj} is known to be contained in $\cI_{\ol{C}}$.

\begin{lem}\label{lem:relative-classical}
For $p> h$, Conjecture~\ref{conj:classical-scheme-conj}  implies Conjecture~\ref{conj:scheme-conj}.  
\end{lem}
\begin{proof}
In view of the ``lower bound'' discussed above, this claim follows by the same argument as in~\cite[Remark~9.4]{achar-hardesty-riche}. 
\end{proof}

\begin{rmk}\label{rmk:rel-classic-A}
Below, for $\mbf{G} = GL_N(\bk)$ with $p > N$, we will give an argument relating these conjectures in the opposite direction: see Corollary~\ref{cor:classical-hr}.
\end{rmk}


\subsection{The type \texorpdfstring{$A$}{A} case}
For the remainder of this section we will assume that $\mbf{G}=GL_N(\bk)$ with $p > N$. Our goal will be to verify 
Conjectures~\ref{conj:classical-scheme-conj} and \ref{conj:scheme-conj} in this setting.
For any $I \subset S$, let $C_I$ denote the Richardson orbit satisfying $\ol{C_I} = G\cdot \fn_I$, and
 recall the well-known fact 
that every nilpotent orbit for $GL_N(\bk)$ is of this form. 

For any $G$-stable closed subset
$Z \subset \cN$, we take 
\[
\Coh^{G \times \Gm}_Z(\cN) \subset \Coh^{G \times \Gm}(\cN),
\quad\text{resp.}\quad
\Db_Z\Coh^{G\times \Gm}(\cN) \subset \Db\Coh^{G\times \Gm}(\cN)
\]
to be the full abelian, resp.~triangulated, subcategory consisting of all objects set-theoretically supported on $Z$. 


\begin{lem}\label{lem:pushforward-surjective}
For any $G$-stable closed subset $Z \subset \cN$, the category $\Db_Z\Coh^{G\times \Gm}(\cN)$ is generated by objects of the form $\pi_{J*} \cF$ with $\cF \in \Db\Coh^{G\times \Gm}(\tcN_J)$,
for various $J\subset S$ such that $C_J \subset Z$. 
\end{lem}
\begin{proof}
Let $D'_Z$ be the full triangulated subcategory of $\Db_Z\Coh^{G \times \Gm}(\cN)$ generated by objects of the form $\pi_{J*}\cF$.  We wish to show that $D'_Z = \Db_Z\Coh^{G \times \Gm}(\cN)$.  We proceed by induction on the number of orbits in $Z$.

It is enough to show that any object $\cG \in \Coh_Z^{G \times \Gm}(\cN)$ lies in $D'_Z$.  Moreover, any such $\cG$ admits a finite filtration whose subquotients are supported on the \emph{reduced} subscheme corresponding to $Z$.  We may therefore assume that $\cG$ itself has reduced scheme-theoretic support.

Choose a subset $I$ such that the $G$-orbit $C_I$ is open in $Z$.  Assume by induction that the lemma is already known for $Z' := Z \smallsetminus C_I$.  We will exhibit an object $\cF \in \Db\Coh^{G \times \Gm}(\tcN_I)$ together with a morphism $\phi: \cG \to \pi_{I*}\cF$ whose cone is supported (set-theoretically) on $Z'$.  This will prove the lemma.

Let $\tilde{C_I} = \pi_I^{-1}(C_I)$, and consider the Cartesian diagram
\begin{equation}\label{eqn:nice-diagram}
\begin{tikzcd}
 \tilde{C_I} \ar[r, hook] \ar[d, "p"'] & \tcN_I \ar[d, "\pi_I"] \\
 C_I \ar[r, hook] & \cN 
\end{tikzcd}
\end{equation}
In this proof, let us write $\pi_{I\circ}$ and $\pi_I^\circ$ for the \emph{underived} push-forward and pullback functors along $\pi_I$.  Set $\cF = \pi_I^\circ \cG \in \Coh^{G \times \Gm}(\tcN_I)$, and then define $\phi: \cG \to \pi_{I*}\cF$ to be the composition of the maps
\[
\cG \to \pi_{I\circ}\pi_I^\circ \cG = \pi_{I\circ} \cF \to \pi_{I*}\cF,
\]
where the first map is the unit of the adjunction, and the last map is the truncation map.  All objects above have scheme-theoretic support contained in the reduced subscheme of $Z$.  To prove the claim, it is enough to show that the restriction of $\phi$ to $C_I$ is an isomorphism.  This restriction is given by
\[
\cG|_{C_I} \to p_\circ p^\circ(\cG|_{C_I}) \to p_* p^\circ(\cG|_{C_I}).
\]
Each of these maps is an isomorphism because the map $p$ in \eqref{eqn:nice-diagram} is an isomorphism: see, for instance,~\cite[Remark~8.8]{jantzen-nilp}).  
\end{proof}

\begin{thm}\label{thm:humphreys-richardson}
Conjecture \ref{conj:scheme-conj} holds for $\mbf{G}=GL_N(\bk)$ and $p > N$. 
\end{thm}
\begin{proof}
By \cite[Proposition~9.1]{achar-hardesty-riche} (combined with Lemma~\ref{lem:silting-parity}), this problem is equivalent to showing that for any orbit $C$ and any $\lambda \in \bX^+_C$, the scheme-theoretic support of the object $\fE_\lambda = \pi_*\tfE_{w_0\lambda}$ is equal to $\ol{C}$.  We already know from~\cite[Theorem~9.3(1)]{achar-hardesty-riche} that the set-theoretic support at least contains $\ol{C}$, so it is enough to show that the scheme-theoretic support is contained in $\ol{C}$.  In fact, we will prove a slightly stronger statement, replacing $\bX^+_C$ by $\bX^+_{\ol{C}}$: we will show that
\begin{equation}\label{eqn:conj-rephrase}
\text{if $\lambda \in \bX^+_{\ol{C}}$, then the scheme-theoretic support of $\fE_\lambda$ is contained in $\ol{C}$.}
\end{equation}

Suppose $C = C_I$ for some $I \subset S$.
 By Lemma~\ref{lem:pushforward-surjective}, $\Db_{\ol{C_I}}\Coh^{G\times \Gm}(\cN)$ is generated by the objects
\[
\pi_{J*}\tfE_{J, \mu}\{n\},
\]
for $n \in \Z$,  $\mu \in \bXpp_J$ where $J$ satisfies $C_J \subset \ol{C_I}$. 
These are silting objects by Corollary~\ref{cor:para-co-t}. Let us set
\[
\bX' := \left\{\lambda \in \bX^+ \mid \text{$\fE_\lambda$ is a summand of some $\pi_{J*}\tfE_{J,\mu}\{n\}$}\right\}.
\]
Since $\pi_J$ factors through the inclusion map $\ol{C_I} \hookrightarrow \cN$, all objects in the collection $\{ \fE_\lambda\}_{\lambda \in \bX'}$ have scheme-theoretic support contained in $\ol{C_I}$.


It follows from \cite[Theorem~9.3(1)]{achar-hardesty-riche} that $\bX' \subseteq \bX^+_{\ol{C}}$. To prove~\eqref{eqn:conj-rephrase}, we must show that
$\bX' = \bX^+_{\ol{C}}$. To accomplish this, we will proceed with a $K$-theoretic argument. The set
\begin{equation}\label{eqn:groth-basis}
\{[\fE_\lambda]\}_{\lambda \in \bXp}
\end{equation}
is a $\Z[t,t^{-1}]$-basis for the Grothendieck group $K\left(\Db\Coh^{G\times \Gm}(\cN)\right)$ (where the action of $t$ is induced by twisting with $\{1\}$).  Moreover, by Lemma~\ref{lem:pushforward-surjective}, for any object $\cF \in \Db_{\ol{C_I}}\Coh^{G \times \Gm}(\cN)$, the class $[\cF]$ lies in the span of the subset
\begin{equation}\label{eqn:groth-c}
\{[\fE_\lambda]\}_{\lambda \in \bX'}.
\end{equation}

Two other bases for $K\left(\Db\Coh^{G\times \Gm}(\cN)\right)$ are
\begin{equation}\label{eqn:groth-basis-2}
\{[\fL_\lambda]\}_{\lambda \in \bXp}
\qquad\text{and}\qquad
\{[\onabla_\lambda]\}_{\lambda \in \bXp}.
\end{equation}
Because the $\fL_\lambda$'s are obtained from the the pre-exceptional set $\{\onabla_\lambda\}_{\lambda \in \bX^+}$ by a recollement construction, the transition matrix between these bases is ``upper-triangular'': we have
\[
[\fL_\lambda] = [\onabla_\lambda] + \sum_{\substack{\mu \in \bX^+ \\ \mu < \lambda}} a_{\lambda,\mu}(t) [\onabla_\mu].
\]
By the same reasoning, the transition matrix between~\eqref{eqn:groth-basis} and the second basis in~\eqref{eqn:groth-basis-2} is also upper-triangular.  It follows that the transition matrix between~\eqref{eqn:groth-basis} and the first basis in~\eqref{eqn:groth-basis-2} is upper-triangular: we have
\begin{equation}\label{eqn:groth-tri}
[\fL_\lambda] = [\fE_\lambda] + \sum_{\substack{\mu \in \bX^+ \\ \mu < \lambda}} b_{\lambda,\mu}(t) [\fE_\mu].
\end{equation}

Now suppose that $\lambda \in \bX^+_{\ol{C_I}}$. Then the object $\fL_\lambda$ is supported on $\ol{C_I}$, so it lies in the span of~\eqref{eqn:groth-c}.  Since $[\fE_\lambda]$ occurs with nonzero coefficient in~\eqref{eqn:groth-tri}, we conclude that $\lambda \in \bX'$, and hence that $\bX' = \bX^+_{\ol{C}}$.
\end{proof}

\begin{cor}\label{cor:classical-hr}
Conjecture \ref{conj:classical-scheme-conj} holds for $\mbf{G}=GL_N(\bk)$ and $p > N$. 
\end{cor}
\begin{proof}
Suppose $\mu \in \bXp$ belongs to the lower closure of $w_\lambda \cdot_p \cC_p$ for some $\lambda \in \bX_C$.  The main result of~\cite{hardesty} already tells us that the $\bk[\cN]$-module $\Ext_{\mbf{G}_{1}}^*(\tilt(\mu), \tilt(\mu))$ has set-theoretic support equal to $\ol{C}$, so we need only prove that its scheme-theoretic support is reduced.  Write $\tilt(\mu)^* \otimes \tilt(\mu)$ as a sum of indecomposable tilting modules, say
\[
\tilt(\mu)^* \otimes \tilt(\mu) = \tilt(\nu_1) \oplus \cdots \oplus \tilt(\nu_k).
\]
We have
\[
\Ext_{\mbf{G}_{1}}^*(\tilt(\mu), \tilt(\mu))
\cong \Ext_{\mbf{G}_{1}}^*(\bk, \tilt(\mu)^* \otimes \tilt(\mu))
\cong \bigoplus_{i = 1}^k \Ext_{\mbf{G}_{1}}^*(\bk, \tilt(\nu_i)).
\]
By Theorem~\ref{thm:humphreys-richardson}, every nonzero term in the last direct sum above has reduced scheme-theoretic support.
\end{proof}

\appendix

\section{Equivariant dg modules}
\label{app:dg}

Let $H$ be an algebraic group over $\bk$, and let $X = \Spec(\mbf{A})$ be an affine $H \times \Gm$-variety over $\bk$.  In other words, $\mbf{A}$ is a graded commutative finitely generated reduced $\bk$-algebra equipped with a rational $H$-action.  We assume throughout that the grading on the ring $\mbf{A}$ is concentrated in nonnegative degrees.

Let $\mbf{A}\lmod_{H \times \Gm}$ be the abelian category of $H \times \Gm$-equivariant $\mbf{A}$-modules, or, equivalently, of graded $H$-equivariant $\mbf{A}$-modules.  Each object $M \in \mbf{A}\lmod_{H \times \Gm}$ comes equipped with an ``internal grading'' $M = \bigoplus_{j \in \Z} M_j$, where each $M_j$ is a rational $H$-module.  For $M \in \mbf{A}\lmod_{H \times \Gm}$, we set
\[
M \la 1 \ra = M \otimes \bk_{-1},
\]
so that the internal grading of $M\la 1\ra$ is given by $(M\la 1\ra)_j = M_{j+1}$.

Given a (possibly infinite) collection $(M^i)_{i \in I}$ of rational $H$-modules, the product vector space $\prod_{i \in I} M^i$ carries a (not necessarily rational) action of the abstract group $H(\bk)$.  A vector $v \in \prod_{i \in I} M^i$ is said to be \emph{rational} if it is contained in a finite-dimensional $H(\bk)$-stable subspace on which $H$ acts algebraically.  The subspace consisting of rational vectors is denoted by
\[
\prodrat_{i \in I} M^i \subset \prod_{i \in I} M^i.
\]
This is a rational $H$-module.  The same notion makes sense for $H \times \Gm$-modules.  If the $M^i$ are objects of $\mbf{A}\lmod_{H \times \Gm}$, then because the $H \times \Gm$-action on $\mbf{A}$ is rational, it is easy to see that $\smallprodrat_{i \in I} M^i$ is an $\mbf{A}$-submodule of $\prod_{i \in I} M^i$.  Thus, in this context, $\smallprodrat_{i \in I} M^i$ is again an object of $\mbf{A}\lmod_{H \times \Gm}$.

Let $\mbf{A}\lmod_{H \times \Gm}^{\fgen}$ be the subcategory consisting of those modules that are finitely generated over $\mbf{A}$. Of course, we may identify
\[
\mbf{A}\lmod_{H \times \Gm} = \QCoh^{H \times \Gm}(X)
\qquad\text{and}\qquad
\mbf{A}\lmod_{H \times \Gm}^{\fgen} = \Coh^{H \times \Gm}(X).
\]

Next, let $C^{++}\mbf{A}\lmod_{H \times \Gm}$ be the category of chain complexes $M = ( \cdots \to M^{-1} \to M^0 \to M^1 \to \cdots)$ that are ``cohomologically doubly bounded below'': that is, the cohomology modules $H^i(M) \in \mbf{A}\lmod_{H \times \Gm}$ vanish for $i \ll 0$, and there is an integer $N$ such that $H^i(M)_j = 0$ for all $j < N$ and all $i \in \Z$.  We do not impose any boundedness conditions on the underlying terms $M^i$.  However, every chain complex in $C^{++}\mbf{A}\lmod_{H \times \Gm}$ is quasi-isomorphic to one whose terms do satisfy such boundedness conditions.  Denote the derived category of $C^{++}\mbf{A}\lmod_{H \times \Gm}$ by
\[
D^{++}\QCoh^{H \times \Gm}(X)
\]

It is well known that the category $\QCoh^{H \times \Gm}(X) = \mbf{A}\lmod_{H \times \Gm}$ has enough injectives. (This follows from the fact that $\QCoh(X)$ has enough injectives, and that the forgetful functor $\QCoh^{H \times \Gm}(X) \to \QCoh(X)$ has an exact right adjoint, namely, the ``averaging'' functor.)  Therefore, any complex in $C^{++}\mbf{A}\lmod_{H \times \Gm}$ is quasi-isomorphic to a bounded-below chain complex of injectives.  (However, the internal grading of injective modules is usually \emph{not} bounded below.)

Next, regard the graded ring $\mbf{A}$ as a dg ring with zero differential. Let $\mbf{A}\ldgmod_H$ denote the category of $H$-equivariant dg modules.  Let $\mbf{A}\ldgmod_H^+ \subset \mbf{A}\ldgmod_H$ be the subcategory consisting of dg modules whose cohomology $H^\bullet(M)$ is bounded below, and then let $\mbf{A}\ldgmod_H^{+,\fgen} \subset \mbf{A}\ldgmod_H^+$ be the subcategory consisting of modules $M$ for which $H^\bullet(M)$ is finitely generated over $\mbf{A}$.  We denote by
\[
D^{+,H}_\QCoh(X)
\qquad\text{and}\qquad
D^{+,\fgen,H}_\Coh(X)
\]
the derived categories of $\mbf{A}\ldgmod_H^+$ and $\mbf{A}\ldgmod_H^{+,\fgen}$, respectively.

We define
\[
\xi: C^{++}\mbf{A}\lmod_{H \times \Gm} \to \mbf{A}\ldgmod_H^+
\]
to be the functor that sends a chain complex $M = ( \cdots \to M^{-1} \to M^0 \to M^1 \to \cdots)$ in $C^{++}\mbf{A}\lmod_{H \times \Gm}$ to the dg module given by
\[
\xi(M)^n = \prodrat_{i+j = n} M^i_j.
\]
It is easy to see that
\begin{equation}\label{eqn:cohom-degrade}
H^n(\xi(M)) \cong \prodrat_{i+j=n} H^i(M)_j \cong \bigoplus_{i+j=n} H^i(M)_j.
\end{equation}
(The ``doubly bounded below'' condition ensures that this direct sum is finite.)  In particular, $\xi$ sends acyclic complexes to acyclic complexes, so it induces functors
\begin{align*}
\xi &: D^{++}\QCoh^{H \times \Gm}(X) \to D^{+,H}_\QCoh(X), \\
\xi &: \Db\Coh^{H \times \Gm}(X) \to D^{+,\fgen,H}_\Coh(X).
\end{align*}
Lastly, we define
\[
D^H_\Coh(X) = 
\begin{array}{c}
\text{the full triangulated subcategory of $D^{+,\fgen,H}_\Coh(X)$ generated by}\\
\text{the essential image of $\xi: \Db\Coh^{H \times \Gm}(X) \to D^{+,\fgen,H}_\Coh(X)$.}
\end{array}
\]
It seems likely that one has $D^H_\Coh(X) = D^{+,\fgen,H}_\Coh(X)$ in many cases of interest, but we will not try to prove this claim here.

\begin{lem}
\begin{enumerate}
\item If $I$ is a bounded-below chain complex of injective objects in $C^{++}\mbf{A}\lmod_{H \times \Gm}$, then $\xi(I)$ is a $K$-injective object in $\mbf{A}\ldgmod_H^+$.\label{it:kinj-xi}
\item Every object $M \in \mbf{A}\ldgmod_H^+$ admits a quasi-isomorphism $M \to I$ where $I \in \mbf{A}\ldgmod_H^+$ is $K$-injective.\label{it:kinj-exist}
\end{enumerate}
\end{lem}
\begin{proof}
\eqref{it:kinj-xi}~Let $I = ( \cdots \to I^{-1} \to I^0 \to I^1 \to \cdots)$ be a bounded-below chain complex of injectives in $C^{++}\mbf{A}\lmod_{H \times \Gm}$.  Assume without loss of generality that $I^i = 0$ for $i < 0$.  Let $M \in \mbf{A}\ldgmod_H$ be an acyclic dg modules, and let $f: M \to \xi(I)$ be a morphism of $H$-equivariant dg $\mbf{A}$-modules.  We must show that $f$ is null-homotopic.

For each $i$, $\xi(I^i)$ is a dg module with zero differential.  The underlying graded $\mbf{A}$-module of $\xi(I)$ (ignoring the differential) is given by
\[
\xi(I) = \prodrat_{i \in \Z} \xi(I^i)[-i],
\]
and its differential $d_{\xi(I)}$ is a product of maps of graded $\mbf{A}$-modules
\[
d_{\xi(I)}^j: \xi(I^j)[-j] \to \xi(I^{j+1})[-j].
\]
Let $d_M: M \to M[1]$ be the differential of $M$.  The map $f: M \to \xi(I)$ is a product of maps
\[
f^j: M \to \xi(I^j)[-j] 
\qquad\text{such that}\qquad
d_{\xi(I)}^j f^j = f^{j+1} d_M.
\]
Below, we will define a collection of maps
\begin{equation}\label{eqn:K-inj}
s^j: M \to \xi(I^j)[-j-1]
\qquad\text{such that}\qquad
d_{\xi(I)}^{j-1} s^{j-1} + s^{j} d_M = f^{j}
\end{equation}
for all $j \in \Z$.  Let $s = \prod_{i \in \Z} s^i: M \to \prod_{i \in \Z} \xi(I^i)[-i-1]$.  Since $M$ is a rational $H \times \Gm$-module, the image of $s$ consists of rational vectors.  That is, we have a map $s: M \to \xi(I)[-1]$ such that $d_{\xi(I)} s + s d_M = f$, as desired.

We define the $s^j$'s by induction on $j$.  For $j < 0$, we must have $s^j = 0$, and the equation in~\eqref{eqn:K-inj} holds trivially.  Suppose now that $s^j$ is defined for all $j < N$ in such a way that~\eqref{eqn:K-inj} holds.  Note that
\[
(f^N - d_{\xi(I)}^{N-1} s^{N-1}) \circ d_M = d_{\xi(I)}^{N-1} f^{N-1} - d_{\xi(I)}^{N-1}(f^{N-1}  - d_{\xi(I)}^{N-2} s^{N-2}) = 0.
\]
It follows that $f^N - d_{\xi(I)}^{N-1} s^{N-1}$ induces a map $g: M/\im d_M \to \xi(I^N)[-N]$.  Since $M$ is acyclic, we have $\im d_M = \ker d_M$, and $M/\im d_M \cong M/\ker d_M \cong \im d_M$.  Regard $\im d_M$ as a submodule of $M[1]$.  Using the fact that $\xi(I^N)$ is an injective module, we can extend $g: \im d_M \to \xi(I^N)[-N]$ to a map $s^N: M[1] \to \xi(I^N)[-N]$.  This map satisfies the equation in~\eqref{eqn:K-inj} by construction.

\eqref{it:kinj-exist}~If $M$ lies in the image of $\xi$, say, $M = \xi(\tilde M)$, then the claim follows from part~\eqref{it:kinj-xi}, since $\tilde M$ admits an injective resolution in $C^{++}\mbf{A}\lmod_{H \times \Gm}$.  For general $M$ with differential $d_M: M \to M[1]$, consider the short exact sequence of $\mbf{A}$-modules $0 \to \ker d_M \to M \to \im d_M \to 0$.  (Here, we are exploiting the fact that $\mbf{A}$ has zero differential.)  This gives rise to a distinguished triangle in $D^{+,H}_\QCoh(X)$.  Now, $\ker d_M$ and $\im d_M$ are both in the image of $\xi$ (since they have zero differential), so they admit $K$-injective resolutions, and hence so does $M$.
\end{proof}

\begin{prop}\label{prop:degrade}
The functor $\xi : \Db\Coh^{H \times \Gm}(X) \to D^H_\Coh(X)$ is a degrading functor with respect to $\la -1\ra[1]$.  In other words:
\begin{enumerate}
\item There is a natural isomorphism $\xi(M) \cong \xi(M\la -1\ra[1])$ such that the map\label{it:graded-ff}
\[
\bigoplus_{n \in \Z} \Hom(M, N\la -n\ra[n]) \simto \Hom(\xi(M), \xi(N))
\]
is an isomorphism for all $M, N \in \Db\Coh^{H \times \Gm}(X)$.
\item The image of $\xi$ generates $D^H_\Coh(X)$ as a triangulated category.\label{it:graded-surj}
\end{enumerate}
\end{prop}
\begin{proof}
Part~\eqref{it:graded-surj} holds by definition.  For part~\eqref{it:graded-ff}, we rely on the well-known fact that $\Db\Coh^{H \times \Gm}(X)$ is equivalent to the full subcategory of $D^{++}\QCoh^{H \times \Gm}(X)$ consisting of objects with bounded, coherent cohomology.  To find $\Hom(M, N \la -n\ra[n])$, we may replace $N$ by an injective resolution in $C^{++}\mbf{A}\lmod_{H \times \Gm}$, and $M$ by a bounded chain complex.  Then $\xi(N)$ is $K$-injective, and the chain complex $R\Hom(M,N)$ is cohomologically doubly bounded below.  The claim follows by applying~\eqref{eqn:cohom-degrade} to $R\Hom(M,N)$.
\end{proof}

\begin{lem}\label{lem:dg-pushpull}
Let $f: X \to Y$ be an $H$-equivariant closed immersion of affine $H$-varieties.
\begin{enumerate}
\item There is a functor $\bar f_*: D^H_\Coh(X) \to D^H_\Coh(Y)$ such that the following diagram commutes:\label{it:dg-push}
\[
\begin{tikzcd}
\Db\Coh^{H \times \Gm}(X) \ar[r, "\xi_X"] \ar[d, "f_*"'] &
  D^H_\Coh(X) \ar[d, "\bar f_*"] \\
\Db\Coh^{H \times \Gm}(Y) \ar[r, "\xi_Y"] &
  D^H_\Coh(Y).
\end{tikzcd}
\]
\item Assume that $\bk[X]$ admits a bounded resolution by $H \times \Gm$-equivariant free $\bk[Y]$-modules. Then there is a functor $\bar f^*: D^H_\Coh(Y) \to D^H_\Coh(X)$ such that the following diagram commutes:\label{it:dg-pull}
\[
\begin{tikzcd}
\Db\Coh^{H \times \Gm}(Y) \ar[r, "\xi_Y"] \ar[d, "f^*"'] &
  D^H_\Coh(Y) \ar[d, "\bar f^*"] \\
\Db\Coh^{H \times \Gm}(X) \ar[r, "\xi_X"] &
  D^H_\Coh(X).
\end{tikzcd}
\]
Moreover, $\bar f^*$ is left adjoint to $\bar f_*$.
\end{enumerate}
\end{lem}
\begin{proof}
Since $f_*$ is an exact functor of coherent sheaves, part~\eqref{it:dg-push} is clear.  For part~\eqref{it:dg-pull}, note that the assumption on the existence of a bounded free resolution of $\bk[X]$ implies that $f^*$ takes values in $\Db\Coh^{H \times \Gm}(X)$ rather than $D^-\Coh^{H \times \Gm}(X)$.

Let us show that $\bar f_*$ admits a left adjoint, i.e., that for $M \in D^H_\Coh(Y)$, the functor $\Hom(M,\bar f_*({-}))$ is representable.  We first claim that this property is preserved under taking cones.  That is, if $M_1 \xrightarrow{\phi} M_2 \to M_3 \to $ is a distinguished triangle in $D^H_\Coh(Y)$, and if $\Hom(M_i,\bar f_*({-}))$ is representable for $i = 1,2$, we claim that it is also representable for $i = 3$.  Let $N_1, N_2 \in D^H_\Coh(X)$ be such that we have natural isomorphisms
\[
\Hom(M_i,\bar f_*({-})) \cong \Hom(N_i, {-})
\qquad\text{for $i = 1,2$.}
\]
These isomorphisms give rise to maps $\eta_i: M_i \to \bar f_*N_i$ for $i = 1,2$.  Moreover, the natural transformation $\Hom(M_2,\bar f_*({-})) \to \Hom(M_1,\bar f_*({-}))$ induced by $\phi$ gives rise to a map $\tilde\phi: N_1 \to N_2$ making the following diagram commute:
\[
\begin{tikzcd}
M_1 \ar[r, "\phi"] \ar[d, "\eta_1"'] & M_2 \ar[d, "\eta_2"] \\
\bar f_*N_1 \ar[r, "\bar f_*\tilde\phi"] & \bar f_*N_2.
\end{tikzcd}
\]
Extend $\tilde\phi: N_1 \to N_2$ to a distinguished triangle $N_1 \xrightarrow{\tilde\phi} N_2 \to N_3 \to$, and then extend the diagram above to a morphism of distinguished triangles
\[
\begin{tikzcd}
M_1 \ar[r, "\phi"] \ar[d, "\eta_1"'] & M_2 \ar[d, "\eta_2"'] \ar[r] 
  & M_3 \ar[d, dashed, "\eta_3"] \ar[r] & {} \\
\bar f_*N_1 \ar[r, "\bar f_*\tilde\phi"] & \bar f_*N_2 \ar[r] & \bar f_*N_3 \ar[r] & .{}
\end{tikzcd}
\]
The new map $\eta_3: M_3 \to \bar f_*N_3$ gives rise to a natural transformation
\[
\Hom(N_3, {-}) \to \Hom(M_3,\bar f_*({-})),
\]
and then a five-lemma argument shows that this map is an isomorphism.  Thus, $\Hom(M_3,\bar f_*({-}))$ is representable.

We now return to the main problem of showing that $\Hom(M,\bar f_*({-}))$ is representable. Since $D^H_\Coh(Y)$ is generated by the image of $\xi_Y$, the previous paragraph tells us that it is enough to show this when $M = \xi_Y(M')$ for some $M' \in \Db\Coh^{H \times \Gm}(Y)$.  The adjunction map $M' \to f_*f^*M'$ gives rise to a map $\xi_Y(M') \to \xi_Y(f_* f^*M') \cong \bar f_* \xi_X(f^*M')$, and this map induces a natural transformation
\[
\Hom(\xi_X(f^*M'),{-}) \to \Hom(M,\bar f_*({-})) .
\]
This map is an isomorphism on objects in the image of $\xi_X$, and hence (in view of Proposition~\ref{prop:degrade}) an isomorphism in general.  We have shown that $\bar f_*$ admits a left adjoint $\bar f^*$, and that if $M = \xi_Y(M')$, then $\bar f^*M \cong \xi_X(f^*M')$.
\end{proof}

We conclude with a lemma on change of equivariance.  If $K \subset H$ is a closed subgroup, there is a forgetful functor $\Res_K^H: \Db\Coh^{H \times \Gm}(X) \to \Db\Coh^{K \times \Gm}(X)$.  If $H/K$ is a projective variety, then this functor has a right adjoint $R\Ind_K^H: \Db\Coh^{K \times \Gm}(X) \to \Db\Coh^{H \times \Gm}(X)$, defined as the composition of the equivalence
\[
\Db\Coh^{K \times \Gm}(X) \cong \Db\Coh^{H \times \Gm}(H \times^K X)
\]
with push-forward along the map $\sigma: H \times^K X \to X$ given by $\sigma(h,x) = h \cdot x$.  (The assumption that $H/K$ is projective implies that $\sigma$ is proper, so that $\sigma_*$ takes values in $\Db\Coh^{H \times \Gm}(X)$ rather than in $D^+\QCoh^{H \times \Gm}(X)$.)

\begin{lem}\label{lem:dg-indres}
Let $K \subset H$ be a closed subgroup such that $H/K$ is projective.  Let $X$ be an $H$-variety.  There are functors $R\Ind_K^H: D^K_\Coh(X) \to D^H_\Coh(X)$ and $\Res_K^H: D^H_\Coh(X) \to D^K_\Coh(X)$ such that the following diagrams commute:
\[
\begin{tikzcd}
\Db\Coh^{K \times \Gm}(X) \ar[r, "\xi_X"] \ar[d, "R\Ind_K^H"'] &
  D^K_\Coh(X) \ar[d, "R\Ind_K^H"] \\
\Db\Coh^{H \times \Gm}(X) \ar[r, "\xi_X"] &
  D^H_\Coh(X),
\end{tikzcd}
\qquad
\begin{tikzcd}
\Db\Coh^{H \times \Gm}(X) \ar[r, "\xi_X"] \ar[d, "\Res_K^H"'] &
  D^H_\Coh(X) \ar[d, "\Res_K^H"] \\
\Db\Coh^{K \times \Gm}(X) \ar[r, "\xi_X"] &
  D^K_\Coh(X).
\end{tikzcd}
\]
Moreover:
\begin{enumerate}
\item The functor $\Res_K^H$ is left adjoint to $R\Ind_K^H$.
\item Let $d = \dim H - \dim K$.  The functor $\Res_K^H({-}) \otimes \bigwedge^d (\mathrm{Lie}(H)/\mathrm{Lie}(K))^*[d]$ is right adjoint to $R\Ind_K^H$.
\end{enumerate}
\end{lem}
\begin{proof}
The existence of functors $R\Ind_K^H$ and $\Res_K^H$ in the dg setting making these diagrams commute is straightforward (using $K$-injective resolutions for the former, and exactness for the latter).  In the coherent sheaf setting, we noted above that $\Res_K^H$ is left adjoint to $R\Ind_K^H$.  In the dg setting, there is at least an obvious natural transformation $\epsilon: \Res_K^H R\Ind_K^H(M) \to M$ for any $K$-injective object $M$.  This natural transformation gives rise to a natural map
\[
\Hom(N, R\Ind_K^H(M)) \to \Hom(\Res_K^H(N), M).
\]
This map is at least an isomorphism when $N$ and $M$ lie in the essential image of $\xi_X$ (by the coherent sheaf version of the adjunction), so it is in fact an isomorphism for all $N$ and $M$.

For the last assertion in the lemma, the dg version again follows from the coherent sheaf version.  The coherent version is a well-known consequence of Serre--Grothendieck duality.  Let us briefly explain how to obtain it.  Let $p: H \times^K X \to X$ be the map given by $p(h,x) = h \cdot x$.  Then $p$ is a smooth, projective bundle over $X$ with fibers isomorphic to $H/K$.  Serre--Grothendieck duality says that for $\cF \in \Db\Coh^H(H \times^K X)$ and $\cG \in \Db\Coh^H(X)$, there is a natural isomorphism
\begin{equation}\label{eqn:serre}
\Hom(\cF, p^*\cG \lotimes \omega_p[d]) \cong \Hom(p_*\cF,\cG),
\end{equation}
where $\omega_p$ is the relative canonical bundle.  (See~\cite[Theorem~III.11.1]{hartshorne} for the nonequivariant version of this statement, and see~\cite[Example~2.16]{ab} for a discussion of how to deduce the equivariant version.)  If we identify $\Coh^H(H \times^K X)$ with $\Coh^K(X)$ (cf.~\eqref{eqn:restric-equiv}), then $p_*$ and $p^*$ are identified with $R\Ind_K^H$ and $\Res_K^H$, respectively, while $\omega_p$ corresponds to the canonical bundle of $H/K$, i.e., the line bundle whose fiber over $1K \in K/H$ is identified with $\bigwedge^d (\mathrm{Lie}(H)/\mathrm{Lie}(K))^*$.  Thus,~\eqref{eqn:serre} yields the desired adjunction.
\end{proof}

\end{document}